\documentclass[10pt,leqno]{article}
\usepackage{csquotes}
\usepackage{mathtools}
\usepackage{amsmath}
\usepackage{amssymb}
\usepackage{amsfonts}
\usepackage{theorem}
\usepackage{color}
\usepackage{empheq}
\usepackage{exscale,relsize}
\usepackage{pifont}

\usepackage[shortlabels]{enumitem}

\definecolor{myblue}{rgb}{.8, .8, 1}

\newcommand{\fenv}[1]%
{\ensuremath{\,\overrightarrow{\operatorname{env}}_{#1}}}
\newcommand{\benv}[1]%
{\ensuremath{\,\overleftarrow{\operatorname{env}}_{#1}}}
\newcommand{\emp}{\ensuremath{\varnothing}}

\newcommand{\scal}[2]{\left\langle{#1},{#2}  \right\rangle}

\newcommand{\RR}{\ensuremath{\mathbb R}}
\newcommand{\RP}{\ensuremath{\left[0,+\infty\right[}}
\newcommand{\RPX}{\ensuremath{\left[0,+\infty\right]}}

\newcommand{\RX}{\ensuremath{\,\left]-\infty,+\infty\right]}}
\newcommand{\RRX}{\ensuremath{\,\left[-\infty,+\infty\right]}}

\newcommand{\NN}{\ensuremath{\mathbb N}}
\newcommand{\Hess}{\ensuremath{\nabla^2\!}}
\newcommand{\oldIDD}{\ensuremath{\operatorname{int}\operatorname{dom}\phi}}
\newcommand{\IDD}{\ensuremath{U}}

\newcommand{\dom}{\ensuremath{\operatorname{dom}}}
\newcommand{\argmin}{\ensuremath{\operatorname*{argmin}}}

\newcommand{\epi}{\ensuremath{\xrightarrow{\operatorname{e}}}}

\newcommand{\menge}[2]{\big\{{#1} \mid {#2}\big\}}

\newcommand{\reli}{\ensuremath{\operatorname{ri}}}
\newcommand{\inte}{\ensuremath{\operatorname{int}}}

\newcommand{\closu}{\ensuremath{\operatorname{cl}}}

\newcommand{\ran}{\ensuremath{\operatorname{ran}}}
\newcommand{\conv}{\ensuremath{\operatorname{conv}}}

\newcommand{\Id}{\ensuremath{\operatorname{Id}}}

\newcommand{\bl}{\ensuremath{\overline{\lambda}}}

\newcommand{\tl}{\ensuremath{\tilde{\lambda}}}

\newcommand{\nphi}{\ensuremath{\nabla{\phi}}}

\newcommand{\GF}{\ensuremath{\Gamma_{0}(\RR^n)}}
\newcommand{\GFO}{\ensuremath{\Gamma_{0}(\RR)}}

\newcommand{\bproj}[1]{\overleftarrow{\thinspace P\thinspace}_%
{\negthinspace\negthinspace #1}}

\newcommand{\bD}[1]{\overleftarrow{\thinspace D\thinspace}_%
{\negthinspace\negthinspace #1}}

\newcommand{\pinf}{\ensuremath{+\infty}}
\newcommand{\sphi}{\ensuremath{S_{\phi}}}

{\begin{list}{}{%
\settowidth{\labelwidth}{\textrm{#1~}}%
\setlength{\leftmargin}{\labelwidth+\labelsep}}}
{\end{list}}
\newtheorem{theorem}{Theorem}[section]
\newtheorem{lemma}[theorem]{Lemma}
\newtheorem{corollary}[theorem]{Corollary}
\newtheorem{proposition}[theorem]{Proposition}
\newtheorem{definition}[theorem]{Definition}
\theoremstyle{plain}{\theorembodyfont{\rmfamily}
}
\theoremstyle{plain}{\theorembodyfont{\rmfamily}
}
\theoremstyle{plain}{\theorembodyfont{\rmfamily}
}
\theoremstyle{plain}{\theorembodyfont{\rmfamily}
\newtheorem{example}[theorem]{Example}}
\newtheorem{fact}[theorem]{Fact}
\theoremstyle{plain}{\theorembodyfont{\rmfamily}
\newtheorem{remark}[theorem]{Remark}}

\newcommand{\averagef}{\ensuremath{\mathcal{P}_{\lambda}^{\phi}(f_{1},f_{2},\alpha)}}
\newcommand{\averagefd}{\ensuremath{\mathcal{P}_{1/\lambda}^{\phi^*}(f_{1}^*,f_{2}^*,\alpha)}}

\newcommand{\lenv}[3]{\ensuremath{\,\overleftarrow{\operatorname{env}}_{#3}^{#1} #2}}
\newcommand{\lprox}[3]{\ensuremath{\,\overleftarrow{\operatorname{prox}}_{#3}^{#1} #2}}

\newcommand{\lhul}[3]{\ensuremath{\,\overleftarrow{\operatorname{hul}}_{#3}^{#1} #2}}
\newcommand{\aprox}[2]{\ensuremath{\,\operatorname{aprox}_{#2}^{#1}}}

\newcommand{\renv}[3]{\ensuremath{\,\overrightarrow{\operatorname{env}}_{#3}^{#1} #2}}
\newcommand{\rprox}[3]{\ensuremath{\,\overrightarrow{\operatorname{prox}}_{#3}^{#1} #2}}

\newcommand{\timess}{\,{\textstyle\mathsmaller{\text{\ding{75}}}}} 

\begin{document}

\title{\sffamily The Bregman proximal average}

\author{
Xianfu
Wang\thanks{Mathematics, Irving K.\ Barber School,
The University of British Columbia Okanagan, Kelowna,
B.C. V1V 1V7, Canada.
Email:
\texttt{shawn.wang@ubc.ca}.} and
Heinz H.\ Bauschke\thanks{Mathematics, Irving K.\ Barber School,
The University of British Columbia Okanagan, Kelowna,
B.C. V1V 1V7, Canada. Email:
\texttt{heinz.bauschke@ubc.ca}.}
}

\date{Submitted August 25, 2021
\\ Revision January 20, 2022} 

\maketitle

\vskip 8mm

\begin{abstract} \noindent We provide a proximal average with repect to a $1$-coercive Legendre
function. In the sense of Bregman distance, the Bregman envelope of
the proximal average is a convex combination of Bregman envelopes
of individual functions. The Bregman proximal mapping of the average is a
convex combination of convexified proximal mappings of individual functions.
Techniques from variational analysis provide the keys for the Bregman proximal average.
\end{abstract}

{\small

}

\section{Introduction}

Starting from the Bauschke, Matou\u{s}kov\'a and Reich \cite{matou}, proximal averages have been
further studied in \cite{bglw,jwp, bauschkel}, and found many applications and generalizations; see, e.g.,
\cite{yu,reid,hare, bw09, planiden, bartz, goebel, kum, Simons}. Bregman proximal mappings play important roles
in the theory of optimization, best approximation, and
the design of optimization
algorithms; see, e.g., \cite{Baus97,ButIus, censor,
bwyy,Sico03,noll,scott,laude,chenkansong,kansong,Butn97,marc}.
An open problem in the literature is to extend the proximal average to the framework of Bregman distances.
 In this paper, we propose a Bregman
proximal average, which unifies and significantly broadens the realm of proximal averages.
It generalizes
the classical proximal average from two perspectives: First the individual functions are not necessarily
convex; second, the proximal mappings are considerably more general.
It is surprising
that the Bregman proximal average has many desirable properties in this generality.
Our main results state that a convex combination of convexified Bregman proximal mappings
is a Bregman proximal mapping, and
that a convex combination of Bregman envelopes is a Bregman envelope.
This extends \cite{bglw,jwp,matou,More65} to the framework of Bregman
distances. Potential algorithmic consequences can be drawn from \cite{Sico03, noll, marc,laude}.

\textbf{Outline of the paper.} The paper is organized as follows. In the remainder of this section we
make  our setting precise. In Section~\ref{s:geodesicscurve}, we collect
a few basic facts and preliminary results on $\phi$-prox-bounded functions,
the Bregman envelopes and proximal maps for possible nonconvex functions,
$\phi$-proximal-hulls, and Combettes-Reyes anisotropic envelopes and proximal mappings.
In Section~\ref{s:clarkemor}, we propose an $\alpha$-weighted Bregman proximal average
with parameter $\mu$ (Bregman proximal average for short)
for $\phi$-prox-bounded proper
lower semicontinuous functions,
and provide its key properties. One important consequence is that a convex combination
of convexified Bregman proximal mappings is a Bregman proximal mapping.
For a general Legendre function $\phi$, even when both functions
are proper lower semicontinuous and convex,
their Bregman proximal average
need not be convex.
Section~\ref{s:whenconvex}
 gives conditions under which the Bregman proximal average is convex.
 To accomplish this we provide
 a Bregman
 version of the Baillon-Haddad theorem and introduce $\nabla\phi$-firmly nonexpansive mappings.
In Section~\ref{s:duality}, we study Fenchel duality properties
of Bregman proximal averages by using
Combettes and Reyes' anisotropic envelopes and  proximity operators. Section~\ref{s:epiconv}
 focuses on the relationships
among arithmetic average, epi-average, and the Bregman proximal average.
It is shown that the proximal hulls of individual functions are the epi-limiting
instances of the Bregman proximal average when $\alpha\downarrow 0$ or $\alpha \uparrow 1$.
It is also shown that
the arithmetic average and epi-average of convexified individual functions
are the limiting instances of the Bregman proximal average
for functions with $+\infty$-prox-bound when $\lambda\downarrow 0$ or $\lambda\uparrow +\infty$.

\textbf{Notation and standing assumptions.}
The notation that we employ is for the most part standard and can be found, for example, in \cite{BC,Rock98,
lewis, urruty1, mordukhovich}; however, a partial
list is provided for the reader's convenience.
Throughout, $\RR^n$ is the standard Euclidean space with inner
product $\scal{\cdot}{\cdot}$ and induced norm $\|\cdot\|$.
The set of proper lower semicontinuous convex functions from
$\RR^n$ to $]-\infty, +\infty]$ is denoted by $\Gamma_{0}(\RR^n)$.
For a set $C\subseteq\RR^n$, its closure, convex hull, closed convex hull, interior and relative interior
are denoted by $\closu{C}$, $\conv C$, $\closu\conv C$, $\inte{C}$ and $\reli{C}$, respectively.
The indicator function of $C$ is $\iota_{C}:\RR^n\rightarrow \RX$ given by $\iota_{C}(x)=0$ if $x\in C$, and
$+\infty$ if $x\not\in C$.
For a function $f:\RR^n\rightarrow\RRX$, its lower semicontinuous hull, convex hull, and closed convex hull
are denoted by $\closu{f}$, $\conv f$ and $\closu\conv f$, respectively. The effective domain of $f$ is
$\dom f:=\menge{x\in\RR^n}{f(x)<-\infty}$. The Fenchel conjugate of $f$ is
$f^{*}(y)=\sup_{x\in\RR^n}(\scal{y}{x}-f(x))$ for every $y\in\RR^n$.
The epi-multiplication of $f$ by $\lambda\in \RP$ is defined by
\begin{equation} \label{e:epimul}
\lambda\timess f := \begin{cases}
\lambda f(\cdot/\lambda), & \text{if $\lambda > 0$;}\\
\iota_{\{0\}}, &\text{if $\lambda =0$.}
\end{cases}
\end{equation}
\begin{definition}
Let
\text{$\phi \in\GF$ be differentiable on $U :=
\oldIDD \neq \emp$.}
The \emph{Bregman distance} associated with $\phi$ is defined by
\begin{equation} \label{eq:D}
D_{\phi}\colon \RR^n \times \RR^n \to \RPX \colon (x,y) \mapsto
\begin{cases}
\phi(x)-\phi(y)-\scal{\nabla \phi(y)}{x-y}, &\text{if}\;\;y\in\IDD;\\
\pinf, & \text{otherwise}.
\end{cases}
\end{equation}
\end{definition}

In this paper, our standing assumptions on $\phi$ are:
\begin{itemize}
\item[\bfseries A1]
$\phi\in\GF$ is of Legendre type,
i.e., $\phi$ is essentially smooth and essentially strictly convex in the
sense of \cite[Section~26]{Rock70}.
\item[\bfseries A2]
$\phi$ is $1$-coercive, i.e., $\displaystyle \lim_{\|x\|\rightarrow
+\infty}\phi(x)/\|x\|=+\infty$. An equivalent requirement is $\dom \phi^*=\RR^n$
(see, e.g., \cite[Theorem~11.8(d)]{Rock98}).
\end{itemize}
Let $f:\RR^n\rightarrow\RX$ be proper and lower semicontinuous. We shall need
two types of envelopes and proximal mappings of $f$: Bregman envelopes and proximal mappings
\cite{kansong,scott}, and
Combettes-Reyes anisotropic envelopes and proximal mappings \cite{com13}.
\begin{definition}
For $\lambda\in ]0,+\infty[$,
the \emph{left Bregman envelope function} to $f$ is defined
by
\begin{equation}
\lenv{\phi}{f}{\lambda}: \RR^n \to \RRX \colon y\mapsto
\inf_{x\in\RR^n}\left(f(x)+\frac{1}{\lambda}D_{\phi}(x,y)\right),
\end{equation}
and the \emph{left Bregman proximal map} of $f$ is
\begin{equation}
\lprox{\phi}{f}{\lambda}\colon \IDD \rightrightarrows  \IDD \colon y\mapsto
\argmin_{x\in\RR^n}\;\: \left(f(x)+\frac{1}{\lambda}D_{\phi}(x,y)\right).
\end{equation}
\end{definition}
The \emph{right Bregman envelope} and \emph{right Bregman proximal mapping} of $f$
are defined analogously and denoted by
$\renv{\phi}{f}{\lambda}$ and $\rprox{\phi}{f}{\lambda}$, respectively.

\begin{definition}
The \emph{Combettes-Reyes anisotropic envelope} of $f$  is defined by
\begin{equation}\label{d:aniso}
f\square\phi: \RR^n \rightarrow\RRX:
x\mapsto \inf_{y\in\RR^n}(f(y)+\phi(x-y)),
\end{equation}
and the \emph{Combettes-Reyes anisotropic proximal map} of $f$ is
$$\aprox{\phi}{f}: \RR^n  \rightrightarrows \RR^n:
x\mapsto \argmin_{y\in\RR^n}(f(y)+\phi(x-y)).$$
\end{definition}
When $\phi(x)=(1/2)\|x\|^2$, $D_{\phi}(x,y)=(1/2)
\|x-y\|^2$, both types of envelopes reduce to the classical Moreau envelope
\cite{Rock98}.
For a general $\phi$, even if $f\in \GF$, the Bregman envelope $\lenv{\phi}{f}{\lambda}$ might not be convex,
although the anisotropic envelope $f\square\phi$ is always convex.

\begin{example}\label{e:nonconvex}
 Let $\lambda :=1$, $f :=\iota_{\{1\}}$ on $\RR$.
\begin{enumerate}
\item\label{i:cubic:env} For $\phi(x)=|x|^3$, we have
$ (\forall y>0)\ \lenv{\phi}{f}{1}(y)=1/3+2y^3/3-y^2$,
which is not convex on $(0,+\infty)$.
\item\label{i:burg:env} For $\phi(x)=-\ln x+x^2/2$ if $x>0$ and $+\infty$ otherwise, we have
$(\forall y>0)\ \lenv{\phi}{f}{1}(y)=\ln y+1/y+(1-y)^2/2-1,$
which is not convex.
\end{enumerate}
\end{example}

\section{Auxiliary results on envelopes and proximal mappings}\label{s:geodesicscurve}
In this section,
we will collect some key facts and preliminary results of Bregman
envelopes and proximal mappings, as well as Combettes-Reyes anisotropic envelope and proximal mappings.
 Throughout this section,
$f:\RR^n\rightarrow\RX$ is proper lower semicontinuous and satisfies
$\dom f\cap \dom \phi\neq \varnothing.$

\subsection{$\phi$-prox-boundedness}
\begin{definition}
A function $f:\RR^n\rightarrow\RX$ is $\phi$-prox-bounded (prox-bounded for short)
if there exists $\lambda>0$
such that $\lenv{\phi}{f}{\lambda}(x)>-\infty$ for some $x\in\RR^n$.
The supremum of all
such $\lambda$ is the threshold $\lambda_{f}$ of the prox-boundedness.
\end{definition}
Prox-boundedness is crucial to ensure pleasant properties for both the
Bregman envelope and proximal mapping.
\begin{fact}\label{f:song}
Let $f:\RR^n\rightarrow\RX$ be proper lower semicontinuos with
prox-bound $\lambda_{f}>0$, and let $0<\lambda<\lambda_{f}$. Then
\begin{enumerate}
\item\label{i:f} $\lenv{\phi}{f}{\lambda}$ is proper lower semicontinuous on $\RR^n$, and continuous on $\IDD$.
\item\label{i:p} 
$\lprox{\phi}{f}{\lambda}$ is nonempty compact valued and upper semicontinuous
on $\IDD$.

\end{enumerate}
\end{fact}
\begin{proof} \ref{i:f}\&\ref{i:p}: See \cite[Theorem 2.2, Corollary 2.2]{kansong},
\cite[Theorem 3.10, 3.16]{chenkansong}.
\end{proof}

The following result extends \cite[Theorem 2.5]{kansong}, in which Kan and Song proved
the result on $\dom f\cap \IDD$ when $\phi$
is strictly convex. As in \cite{borwein-vand}, an essentially strictly convex function need
not be strictly convex.

\begin{proposition} \label{i:c}
Let $f:\RR^n\rightarrow\RX$ be proper lower semicontinuos with
prox-bound $\lambda_{f}>0$, and let $0<\lambda<\lambda_{f}$. Then
$(\forall x\in U)\ \lim_{\lambda\downarrow 0}\lenv{\phi}{f}{\lambda}(x)=f(x).$
\end{proposition}

\begin{proof}
 In view of \cite[Theorem 3.7(iv)]{Baus97}, for $y\in\IDD$,
$D_{\phi}(x,y)=0\Leftrightarrow x=y$.  When $y
\in\dom f\cap \IDD$, the same arguments as in the proof of \cite[Theorem 2.5]{kansong}
shows that $\lim_{\lambda\downarrow 0}\lenv{\phi}{f}{\lambda}(x)=f(x)$.
When $y\in \IDD\setminus \dom f$, $f(y)=+\infty$, it suffices to show that for every
sequence $(\lambda_{k})_{k\in\NN}$ with $\lambda_{k}\downarrow 0$ we have
\begin{equation}\label{e:outside}
\lim_{k\rightarrow\infty}\lenv{\phi}{f}{\lambda_{k}}(y)=+\infty.
\end{equation}
 Indeed, following
the proof of \cite[Theorem 2.5]{kansong} we have a sequence
$(w_{k})_{k\in\NN}$ such that $w_{k}\rightarrow\bar{w}$ and
$f(w_{k})+\frac{1}{\lambda_{k}}D_{\phi}(w_{k},y)=\lenv{\phi}{f}{\lambda_{k}}(y).$
If $\bar{w}\neq y$, then $D_{\phi}(\bar{w},y)>0$ and
\begin{align}
\liminf_{k\rightarrow\infty} \lenv{\phi}{f}{\lambda_{k}}(y) &\geq
\liminf_{k\rightarrow\infty} f(w_{k})+\liminf_{k\rightarrow\infty}\frac{1}{\lambda_{k}}D_{\phi}(w_{k},y)\\
&\geq f(\bar{w})+D_{\phi}(\bar{w},y)/0^{+}=+\infty.
\end{align}
If $\bar{w}=y$, then
$\liminf_{k\rightarrow\infty}\lenv{\phi}{f}{\lambda_{k}}(y)\geq \liminf_{k\rightarrow\infty}f(w_{k})\geq
f(\bar{w})
=+\infty.$
Hence, \eqref{e:outside} holds.
\end{proof}

The threshold of prox-boundedness has the following useful characterization, which complements
\cite[Proposition 3.1]{laude}.
\begin{proposition}\label{p:bound} The following hold:
\begin{enumerate}
\item\label{i:boundl}
 If $f$ is prox-bounded with threshold $\lambda_{f}>0$, then
for every $\lambda\in ]0, \lambda_{f}[$ the function
$f+\frac{1}{\lambda}\phi$ is bounded below. Consequently, for every $\lambda\in ]0, \lambda_{f}[$
the function $f+\frac{1}{\lambda}\phi$ is $1$-coercive.
\item\label{i:proxl} If there exists $\ell>0$ such that for every $\lambda\in ]0,\ell[$
the function
$f+\frac{1}{\lambda}\phi$ is bounded below, then $\lambda_{f}\geq \ell$.
\item\label{i:same} Define
$\ell_{f} :=\sup\left\{\ell>0: (\forall \lambda\in ]0,\ell[)\
\inf\bigg( f+\frac{1}{\lambda}\phi\bigg)>-\infty\right\}.$
Then $\ell_{f}=\lambda_{f}.$
\end{enumerate}
\end{proposition}
\begin{proof} We follow the proof idea of \cite[Proposition 3.5]{laude}.
Because $\phi$ is $1$-coercive and Legendre, we have $\nabla \phi^*(0)\in U$.

\ref{i:boundl}: For every $\lambda\in ]0,\lambda_{f}[$,
one has $\lenv{\phi}{f}{\lambda}(\nabla \phi^*(0))>-\infty$. This gives
$$(\forall w\in\RR^n)\ f(w)+\frac{1}{\lambda}\phi(w)\geq \frac{1}{\lambda}\phi(\nabla \phi^*(0))+\lenv{\phi}{f}{\lambda}(\nabla \phi^*(0)),$$
which implies $f+\frac{1}{\lambda}\phi$ is bounded below. Now every $\tilde{\lambda}\in ]0, \lambda_{f}[$ and take
$\lambda\in ]\tilde{\lambda},\lambda_{f}[$. Since $f+\frac{1}{\lambda}\phi$ is bounded below,
$1/\lambda<1/\tilde{\lambda}$, $\phi$ is $1$-coercive, and
$f+\frac{1}{\tilde{\lambda}}\phi=f+\frac{1}{\lambda}\phi+\big(\frac{1}{\tilde{\lambda}}-\frac{1}{\lambda}\bigg)\phi,
$
we conclude that $f+\frac{1}{\tilde{\lambda}}\phi$ is $1$-coercive.

\ref{i:proxl}: For every $\lambda\in ]0,\ell[$, we have
$\lenv{\phi}{f}{\lambda}(\nabla \phi^*(0))=
\inf_{w\in\RR^n}\bigg(f(w)+\frac{1}{\lambda}\phi(w)\bigg)-\frac{1}{\lambda}\phi(\nabla\phi^*(0))
>-\infty
$
by the assumption.  Hence $\lambda_{f}\geq \ell$.

\ref{i:same}: Combine \ref{i:boundl} and \ref{i:proxl}.
\end{proof}
\begin{corollary}\label{c:proxbound}
 If a function $f:\RR^n\rightarrow\RX$ is bounded below by a linear function,
then $\lambda_{f}=+\infty$. In particular, this holds when $f\in \GF$.
\end{corollary}
\begin{proof} This is because that $\phi$ is $1$-coercive.  When $f\in \GF$,
$f$ is bounded below by a linear functional by the Brondsted-Rockafellar
theorem, see, e.g., \cite[Theorem 16.58]{BC}.
\end{proof}

\subsection{Properties of the Bregman envelope and proximal mapping}

The following is a slightly refined version of \cite[Theorem 2.4]{kansong}.
\begin{fact}\label{f:env:conj}
Let $f:\RR^n\rightarrow\RX$ be proper lower semicontinuos with
prox-bound $\lambda_{f}>0$, and let $0<\lambda<\lambda_{f}$.
Then the following hold:
\begin{enumerate}
\item \label{i:env:def}
$\lenv{\phi}{f}{\lambda}=\left(\frac{\phi^*-(\lambda f+\phi)^*}{\lambda}\right)\circ
\nabla \phi,
\text{ and }$
\begin{equation}\label{e:env2}
(\lambda f+\phi)^*=\phi^*-\lambda \lenv{\phi}{f}{\lambda}
\circ\nabla\phi^*.
\end{equation}
\item\label{i:env:loc}
If $\nabla \phi$ is locally Lipschitz on $U$, then
$\lenv{\phi}{f}{\lambda}$ is
locally Lipschitz on $U$.
\end{enumerate}
\end{fact}
\begin{proof}
\ref{i:env:def}: The calculation given in \cite[Theorem 2.4]{kansong}
applies to every function $f$.
\ref{i:env:loc}: This is given by \cite[Theorem 2.4]{kansong}.
\end{proof}
\begin{remark} When $\lambda=1$ and $f\in\GF$,  in \cite{com13} Combettes and Reyes used the
notation $f\diamond\phi$ for $\lenv{\phi}{f}{\lambda}$, and \cite[Theorem 1(i)]{com13} coincides
with \eqref{e:env2}.
\end{remark}

\begin{corollary}\label{c:breg:convex}
Let $f:\RR^n\rightarrow\RX$ be proper lower semicontinuos with
prox-bound $\lambda_{f}>0$, and let $0<\lambda<\lambda_{f}$.
If $\lambda f+\phi$ is convex, then
$\lambda f+\phi=(\phi^*-\lambda \lenv{\phi}{f}{\lambda}\circ\nabla\phi^*)^*.$
Consequently,
$$f=\frac{(\phi^*-\lambda \lenv{\phi}{f}{\lambda}
\circ\nabla\phi^*)^*-\phi}{\lambda} \text{ on $\dom \phi$.} $$
\end{corollary}

Let $\hat{\partial}$, $\partial$, and $\partial_{C}$ denote the Fr\'echet subdifferential,
Mordukhovich limiting subdifferential, and Clarke subdifferential, respectively;
see, e.g., \cite{Rock98,mordukhovich,Clarke}.
While
$\hat{\partial}$, $\partial$ and $\partial_{C}$ are different in general,
it is well-known that they
coincide for proper lower
semicontinuous convex functions. The following fact by Kan and Song shows that
the Fr\'echet, limiting, and Clarke subdifferential coincide
for $-\lenv{\phi}{f}{\lambda}$ and they can be found
by using the convex hull of the Bregman proximal mapping
of $f$.

\begin{fact}\emph{\cite[Theorem 3.1]{kansong}}\label{f:derivative}
Let $f:\RR^n\rightarrow\RX$ be proper lower semicontinuos with
prox-bound $\lambda_{f}>0$, and let $0<\lambda<\lambda_{f}$.
Suppose $\phi$ is second-order continuously differentiable on $U$.
Then on $\IDD$ the function $-\lenv{\phi}{f}{\lambda}$
is Clarke regular, and satisfies
$$(\forall x\in\IDD)\ \hat{\partial}(-\lenv{\phi}{f}{\lambda})(x)=\partial_{C}(-\lenv{\phi}{f}{\lambda})(x)
=\frac{1}{\lambda}\nabla^{2}\phi(x)[\conv(\lprox{\phi}{f}{\lambda}(x))-x].$$
\end{fact}

%
%

The following result establishes the relationship between the Bregman proximal mapping
of $f$
and the limiting subdifferential of $f$.
\begin{proposition} \label{p:prox:e}
Let $f:\RR^n\rightarrow\RX$ be proper lower semicontinuos with
prox-bound $\lambda_{f}>0$, and let $0<\lambda<\lambda_{f}$.
Then the following hold:
\begin{enumerate}
\item\label{i:prox:in}
 $\lprox{\phi}{f}{\lambda}\subseteq [\partial(\phi+\lambda  f)]^{-1}\circ\nabla\phi.$
If
\begin{equation}\label{e:cq}
\partial^{\infty}f(y)\cap -N_{\dom\phi}(y)=\{0\} \text{ for every $y\in \dom\phi$,}
\end{equation}
 then
$\lprox{\phi}{f}{\lambda}\subseteq (\nabla \phi+\lambda\partial f)^{-1}\circ\nabla\phi.$
\item\label{i:convexf}
 If $\lambda f+\phi$ is convex, then $(\forall x\in\RR^n)\ \lprox{\phi}{f}{\lambda}(x)$ is convex and closed,
 and
$\lprox{\phi}{f}{\lambda}=[\partial(\phi+\lambda f)]^{-1}\circ\nabla\phi.$
If, in addition, \eqref{e:cq} holds and $f$ is Clarke regular, then
$\lprox{\phi}{f}{\lambda}=(\nabla \phi+\lambda\partial f)^{-1}\circ\nabla\phi.$
\item\label{i:both:conv}
If $f$ is convex, and $(\dom f)\cap U\neq\varnothing$, then
\begin{equation}\label{e:l:prox}
\lprox{\phi}{f}{\lambda}=(\nabla \phi+\lambda\partial f)^{-1}\circ\nabla\phi=
\bigg(\frac{1}{\lambda}\nabla\phi+\partial f\bigg)^{-1}\circ\bigg(\frac{1}{\lambda}\nabla\phi\bigg).
\end{equation}
Moreover, $\lprox{\phi}{f}{\lambda}$ is continuous on $U$.
\end{enumerate}
\end{proposition}

\begin{proof} Consider
the function
$x\rightarrow \frac{1}{\lambda}\big(\lambda f(x)+
\phi(x)-\phi(y)-\scal{\nabla \phi(y)}{x-y}\big).
$

\ref{i:prox:in}: $x\in \lprox{\phi}{f}{\lambda}(y)$ implies
$0\in \partial (\lambda  f+ \phi)(x)-\nabla\phi(y),$
so $x\in [\partial (\lambda  f+ \phi)]^{-1}(\nabla\phi (y))$.
When \eqref{e:cq} holds, $\partial (\lambda  f+\phi)\subseteq \lambda\partial f+\nabla\phi$.

\ref{i:convexf}: The convexity of $\lambda f+\phi$ ensures
that $x\in \lprox{\phi}{f}{\lambda}(y)$ if and only if
$0\in \partial (\lambda  f+ \phi)(x)-\nabla\phi(y),$
which implies $\lprox{\phi}{f}{\lambda}(y)=[\partial (\lambda f+\phi)]^{-1}(\nabla\phi(y))$.
For each fixed $y\in U$, being the set of minimizers of convex function
$x\mapsto \lambda f(x)+\phi(x)-\scal{\nabla\phi(y)}{x-y}$, $\lprox{\phi}{f}{\lambda}(y)$ is
convex and closed. When \eqref{e:cq} holds and $f$ is Clarke regular,
$\partial (\lambda  f+\phi)=\lambda\partial f+\nabla\phi$
by \cite[Proposition 8.12, Corollary 10.9]{Rock98}.

\ref{i:both:conv}: Under the assumption
$(\dom f)\cap U\neq\varnothing$ (instead of \eqref{e:cq})
 the calculus rule
$\partial(\phi+\lambda f)=\partial \phi +\lambda\partial f$ holds for convex functions $\phi$ and $f$;
see, e.g.,\cite[Corollary 16.48(ii)]{BC}. Hence
\eqref{e:l:prox} follows from \ref{i:convexf}.
Because $\phi+\lambda f$ is
essentially strictly convex and $1$-coercive, the conjugate
$(\phi+\lambda f)^*$ is full domain and differentiable, so
$\nabla (\phi+\lambda f)^*=(\nabla\phi+\lambda\partial f)^{-1}$ is
continuous on $\RR^n$, see, e.g., \cite[Corollary 25.5.1]{Rock70}.
As $\nabla\phi$ is continuous on $U$,
we obtain that $\lprox{\phi}{f}{\lambda}$ is continuous
on $U$.
\end{proof}

\begin{remark}
Proposition~\ref{p:prox:e}\ref{i:prox:in} is a
pointwise version reformulation of
\cite[Lemma 3.3]{laude}. See also \cite{Sico03,scott}
for $\lenv{\phi}{f}{\lambda}$ and $\lprox{\phi}{f}{\lambda}$ when $f\in\GF$.
In \cite{bui}, $\lprox{\phi}{f}{\lambda}$ is called as a warped proximity operator.
\end{remark}

Our next result provides a
 connection between $\partial (\lambda f+\phi)^*$ and $\lprox{\phi}{f}{\lambda}$.
\begin{proposition}\label{t:grad}
Let $f:\RR^n\rightarrow\RX$ be proper lower semicontinuos with
prox-bound $\lambda_{f}>0$, let $0<\lambda<\lambda_{f}$, and let
$\nabla^{2}\phi(x)$ be invertible for every $x\in U$. Then
\begin{equation}\label{e:g:connect}
\partial(\lambda f+\phi)^*=\conv\lprox{\phi}{f}{\lambda}\circ\nabla\phi^* \text{ on $U$.}
\end{equation}
Hence, $\conv\lprox{\phi}{f}{\lambda}\circ\nabla\phi^*$ is always maximally monotone.
 If, in addition, $\lambda f+\phi$ is
convex, then
$\partial(\lambda f+\phi)^*=\lprox{\phi}{f}{\lambda}\circ\nabla\phi^*.$
\end{proposition}
\begin{proof} By Fact~\ref{f:env:conj}, we get
$(\forall x\in U)\ [(\lambda f+\phi)^*-\phi^*](\nabla\phi(x))=-\lambda \lenv{\phi}{f}{\lambda}(x).$
Taking subdifferential both sides, by the chain rule \cite[Theorem 10.6]{Rock98} and Fact~\ref{f:derivative}, we have
$$(\forall x\in U)\ \nabla^2\phi(x) \partial[(\lambda f+\phi)^*-\phi^*](\nabla\phi(x))=\nabla^2\phi(x)[\conv \lprox{\phi}{f}{\lambda}(x)-x]$$
from which
\begin{equation}\label{e:difference}
(\forall x\in U)\ \partial[(\lambda f+\phi)^*-\phi^*](\nabla\phi(x))=\conv \lprox{\phi}{f}{\lambda}(x)-x,
\end{equation}
because $\nabla^2\phi(x)$ is invertible by the assumption.
By the sum rule \cite[Exercise 10.10]{Rock98},
$$\partial[(\lambda f+\phi)^*-\phi^*]=\partial (\lambda f+\phi)^*-\nabla\phi^*=\partial (\lambda f+\phi)^*
-(\nabla\phi)^{-1}.$$
Thus,
$ (\forall x\in U)\ \partial (\lambda f+\phi)^*(\nabla\phi(x))=\conv \lprox{\phi}{f}{\lambda}(x)$ by \eqref{e:difference}.
When $\lambda f+\phi$ is convex, $\lprox{\phi}{f}{\lambda}$ is convex-valued by Proposition~\ref{p:prox:e}\ref{i:convexf},
so $\conv$ is superfluous in \eqref{e:g:connect}.
\end{proof}


\subsection{$\lambda$-$\phi$-proximal hull}
The $\lambda$-$\phi$-proximal hull defined below extends the classical
proximal hull \cite[Example 1.44]{Rock98} ($\phi(x)=(1/2)\|x\|^2$), which is a special case
of the Lasry-Lions envelope \cite{attouch}, \cite[Example 1.46]{Rock98}.
\begin{definition} For a function $f:\RR^n\rightarrow \RX$ and $\lambda>0$, the
$\lambda$-$\phi$-proximal hull ($\lambda$-proximal hull for short) of $f$ is
the function $\lhul{\phi}{f}{\lambda}:
\RR^n \rightarrow [-\infty, +\infty]$ defined as the pointwise supremum of the collection
of all the functions of the form
$x\mapsto c-\frac{1}{\lambda}D_{\phi}(x,w)$ that are majorized by $f$,
where $c\in\RR, w\in U$.
\end{definition}
\begin{proposition}\label{p:phi:hull}
The following hold:
\begin{enumerate}
\item\label{i:h:s}
 $\lhul{\phi}{f}{\lambda}=-
\renv{\phi}{(-\lenv{\phi}{f}{\lambda})}{\lambda}$, i.e.,
$(\forall x\in\RR^n)\ \lhul{\phi}{f}{\lambda}(x)=\sup_{w\in U}\bigg(\lenv{\phi}{f}{\lambda}(w)-\frac{1}{\lambda}D_{\phi}(x,w)
\bigg).$
Moreover, $\lenv{\phi}{(\lhul{\phi}{f}{\lambda})}{\lambda}=
\lenv{\phi}{f}{\lambda}.$
\item\label{i:h:c}
 $\lhul{\phi}{f}{\lambda}=\big(f+\frac{1}{\lambda}\phi\big)^{**}-\frac{1}{\lambda}\phi$,
 where we use the convention $\infty-\infty=\infty$. If, in addition, $f+\frac{1}{\lambda}\phi
 \in\GF$, then
 $\lhul{\phi}{f}{\lambda}=f+\iota_{\dom\phi}$.
\item\label{i:h:e}
 $f\geq \lhul{\phi}{f}{\lambda}\geq \lenv{\phi}{f}{\lambda}$ on $U$.
\end{enumerate}
\end{proposition}
\begin{proof}
\ref{i:h:s}:
Denote $\phi_{c,w}=c-\frac{1}{\lambda}D_{\phi}(\cdot,w)$. Then
$\phi_{c,w}\leq f$ if and only if $(\forall x\in \RR^n)\
c\leq f(x)+\frac{1}{\lambda}D_{\phi}(x,w)$, which means
$c\leq \lenv{\phi}{f}{\lambda}(w).$ Therefore,
$\lhul{\phi}{f}{\lambda}$ can be viewed as the pointwise supremum
of the collection of the functions of the form
$\lenv{\phi}{f}{\lambda}(w)-\frac{1}{\lambda}D_{\phi}(x,w)$ with $w\in U$.
The collection of $\phi_{c,w}$ with $\phi_{c,w}\leq f$
is the same as the collection of all $\phi_{c,w}$ with $\phi_{c,w}
\leq \lhul{\phi}{f}{\lambda}$.
Since
$$\lenv{\phi}{f}{\lambda}(w)=\sup\left\{c\middle|\ (\forall x\in\RR^n)\ c\leq f(x)+\frac{1}{\lambda}D_{\phi}(x,w)\right\},$$
$$\lenv{\phi}{(\lhul{\phi}{f}{\lambda})}{\lambda}(w)=\sup\left\{c\middle|\ (\forall x\in\RR^n)\ c\leq \lhul{\phi}{f}{\lambda}(x)+\frac{1}{\lambda}D_{\phi}(x,w)\right\},$$
this reveals that
$\lenv{\phi}{f}{\lambda}=\lenv{\phi}{(\lhul{\phi}{f}{\lambda})}{\lambda}.$

\ref{i:h:c}: By Fact~\ref{f:env:conj} and \ref{i:h:s}, we have
 $\lhul{\phi}{f}{\lambda}(x)=$
\begin{align}
 &
\sup_{w\in\RR^n}\left[\bigg(\frac{1}{\lambda}\phi^*-\frac{1}{\lambda}(\lambda f+\phi)^*\bigg)\circ\nabla\phi(w)-\frac{1}{\lambda}D_{\phi}(x,w)\right]\\
&=
\sup_{w\in\IDD}\left[\bigg(\frac{1}{\lambda}\phi^*-\frac{1}{\lambda}(\lambda f+\phi)^*\bigg)(\nabla\phi(w))+\frac{1}{\lambda}\phi(w)+
\frac{1}{\lambda}\scal{\nabla\phi(w)}{x-w}
\right]-\frac{1}{\lambda}\phi(x)\\
& =\frac{1}{\lambda}
\sup_{w\in\IDD}\left[-(\lambda f+\phi)^*(\nabla\phi(w))+\phi^*(\nabla\phi(w))+
\phi(w)-\scal{\nabla\phi(w)}{w}+\scal{\nabla\phi(w)}{x}
\right]-\frac{1}{\lambda}\phi(x)\label{e:h:one}\\
& = \frac{1}{\lambda}
\sup_{w\in\IDD}\left[-(\lambda f+\phi)^*(\nabla\phi(w))+\scal{\nabla\phi(w)}{x}
\right]-\frac{1}{\lambda}\phi(x)\label{e:h:two}\\
&= \frac{1}{\lambda}(\lambda f+\phi)^{**}(x)-\frac{1}{\lambda}\phi(x)
=\bigg(f+\frac{1}{\lambda}\phi\bigg)^{**}(x)-\frac{1}{\lambda}\phi(x),
\end{align}
in which we used $\phi^*(\nabla\phi(w))+\phi(w)=\scal{\nabla\phi(w)}{w}$
 in \eqref{e:h:one},
  and
$\ran\nabla\phi=\RR^n$ in \eqref{e:h:two}. When $f+\frac{1}{\lambda}\phi\in \GF$,
the Fenchel-Moreau biconjugate theorem
\cite[Theorem 13.37]{BC} gives $\big(f+\frac{1}{\lambda}\phi\big)^{**}=
f+\frac{1}{\lambda}\phi.$

\ref{i:h:e}: This follows from \ref{i:h:s} and \ref{i:h:c}.
\end{proof}

\subsection{Properties of the Combettes-Reyes envelope and proximal mapping}
The following result refines and complements some results of \cite{com13}.
\begin{proposition}\label{p:aniso:env} Let $f\in\GF$.
Then the following hold:
\begin{enumerate}
\item\label{i:aniso:d}
 $\dom f\square \phi=\dom f+\dom\phi$, and
 $f\square \phi \in\GF$ is essentially smooth, so continuously
 differentiable on $\inte\dom (f\square \phi)=\dom f +\IDD$.
\item\label{i:aniso:s} $\dom\aprox{\phi}{f}=
\dom f+\dom\phi$. For every $x\in \dom f+\dom\phi$, $\aprox{\phi}{f}(x)$
is single-valued.
\item\label{i:aniso1} $\aprox{\phi}{f}$ is continuous on $\dom f+\IDD$. Moreover,
\begin{equation}\label{e:exp}
(\forall x\in \dom f +\IDD) \aprox{\phi}{f}(x)=(\Id+\nabla\phi^*\circ \partial f)^{-1}(x).
\end{equation}
\item\label{i:aniso2}
 $\argmin f \cap U =\{x\in U:\ \aprox{\phi^*}{f^*}(\nabla\phi (x))=0\}.$
\item\label{i:aniso3}
 If $\phi$ is nonnegative, and $\phi(0)=0$, then
 \begin{equation}\label{e:function}
 f\geq f\square\phi,\quad \inf f= \inf (f\square\phi), \text{ and }
 \end{equation}
 \begin{equation}\label{e:set}
 \argmin f=\argmin (f\square\phi).
 \end{equation}
\end{enumerate}
\end{proposition}
\begin{proof}
\ref{i:aniso:d}: Apply \cite[Proposition 12.6(ii)]{BC} for $\dom f\square \phi$. Because
$f\in\GF$ and $\phi$ is essentially smooth with $\dom\phi^*=\RR^n$, \cite[Corollary 26.3.2]{Rock70}
shows that $f\square \phi \in\GF$ is essentially smooth. Moreover, $\inte\dom (f\square \phi)=\dom f+\IDD$
because $\dom f+\IDD\subseteq \reli\dom(f\square \phi)=\reli\dom f+\reli\dom\phi\subseteq\dom f+\IDD$.

\ref{i:aniso:s}: For every $x\in\dom f +\dom\phi$, the function
$y\mapsto f(y)+\phi(x-y)$ is in $\GF$, essentially strictly convex and $1$-coercive,
so it has a unique minimizer.

\ref{i:aniso1}: Let $x\in\dom f+\IDD$. We show that $\aprox{\phi}{f}$ is continuous at $x$.
Let $(x_{k})_{k\in\NN}$ be an arbitrary sequence in $\dom f+\IDD$  such that $x_{k}\rightarrow x$, and
let $y_{k}:=\aprox{\phi}{f}(x_{k})$. It suffices to show $y_{k}\rightarrow \aprox{\phi}{f}(x)$.
First we show that $(y_{k})_{k\in\NN}$ is bounded.
Suppose not, after passing to a subsequence and relabelling, we can assume $\|y_{k}\|\rightarrow\infty$.
Now $f\in\GF$ ensures that $f$ possesses a continuous minorant, say,
$f\geq \scal{u}{\cdot}+\eta$ for some $u\in\RR^n$ and $\eta\in\RR$. By \ref{i:aniso:d}
and $(f\square\phi) (x_{k})=f(y_{k})+\phi(x_{k}-y_{k})$, we get
\begin{align*}
(f\square\phi)(x) &\leftarrow (f\square\phi)(x_{k})=f(y_{k})+\phi(x_{k}-y_{k})\\
&\geq \scal{u}{y_{k}}+\eta+\phi(x_{k}-y_{k})\geq \|y_{k}\|(-\|u\|+\phi(x_{k}-y_{k})/\|y_{k}\|)
+\eta\\
&\rightarrow +\infty,
\end{align*}
which is impossible. Hence, $(y_{k})_{k\in\NN}$ is bounded. Next we show that
$(y_{k})_{k\in\NN}$ has a unique subsequential limit, namely, $\aprox{\phi}{f}(x)$.
Indeed, let $(y_{k_{l}})_{l\in\NN}$ be a convergent subsequence of $(y_{k})_{k\in\NN}$
with a limit $y\in\RR^n$. Since $f\square \phi$ is continuous on
$\dom f+\IDD$ by \ref{i:aniso:d}, we have
$(f\square\phi)(x)=\lim_{l\rightarrow\infty} (f\square\phi)(x_{k_{l}})=\lim_{l\rightarrow\infty}(f(y_{k_{l}})+\phi(x_{k_{l}}-y_{k_{l}}))\geq\liminf_{l\rightarrow\infty}
f(y_{k_{l}})+\liminf_{l\rightarrow\infty}\phi(x_{k_{l}}-y_{k_{l}})
\geq f(y)+\phi(x-y)\geq (f\square
\phi)(x)$,
from which $f(y)+\phi(x-y)=(f\square\phi)(x)$, and so $y=\aprox{\phi}{f}(x)$ by \ref{i:aniso:s}.
We conclude that
$\aprox{\phi}{f}$ is continuous at $x$. In turn,
\eqref{e:exp} follows from \cite[Proposition 6]{com13}.

\ref{i:aniso2}: We have $0\in\partial f(x)\Leftrightarrow x\in\partial f^*(0) \Leftrightarrow
\nabla\phi(x)\in \nabla\phi \circ \partial f^*(0)\Leftrightarrow 0\in (\Id+\nabla\phi\circ \partial f^*)^{-1}
(\nabla\phi(x))\Leftrightarrow 0=(\Id+\nabla\phi\circ\partial f^*)^{-1}
(\nabla\phi(x))=\aprox{\phi^*}{f^*}(\nabla\phi(x))$,
because $ (\Id+\nabla\phi\circ\partial f^*)^{-1}$ is single-valued and \ref{i:aniso1}.

\ref{i:aniso3}: \eqref{e:function} follows from \eqref{d:aniso}.
 To see \eqref{e:set},
let $x\in\argmin f$. By $\phi\geq 0$ and \eqref{e:function}, we have
$\inf (f\square \phi)=\inf f=f(x)\geq (f\square\phi)(x)$, so
$x\in\argmin (f\square\phi)$. Conversely, let $x\in\argmin (f\square\phi)$. Because $y\mapsto
f(y)+\phi(x-y)$ is $1$-coercive, there exists $y\in\RR^n$ such that
$\inf f=\inf (f\square\phi)=(f\square\phi)(x)=
f(y)+\phi(x-y)\geq \inf f,$
which implies $f(y)=\inf f$ and $\phi(x-y)=0$. Because $\phi\geq 0$, $\phi(0)=0$,
$\phi$ is essentially strictly convex, $\phi$ must have a unique minimizer at $0$,
so $x=y$. Hence $x\in\argmin f$. Altogether, $\argmin f=\argmin (f\square\phi)$.
\end{proof}

Our last result in this subsection expresses proximal mappings by anisotropic proximal mappings.
\begin{proposition}\label{t:prox:aniso}
Suppose that $f\in\GF$ and $(\reli\dom f)\cap U\neq\varnothing$. Then for $\lambda>0$ one has
$$(\forall x\in U)\ \lprox{\phi}{f}{\lambda}(x)=\nabla\phi^*\bigg(\nabla\phi(x)-\lambda\aprox{1/\lambda\timess\phi^*}{f^*}
\big(\nabla\phi(x)/\lambda\big)\bigg).$$
Consequently,
$(\forall x\in U)\ \nabla\phi \big(\lprox{\phi}{f}{\lambda}(x)\big)+\lambda\aprox{1/\lambda\timess\phi^*}{f^*}
\big(\nabla\phi(x)/\lambda\big)=\nabla\phi(x).$
\end{proposition}
\begin{proof} By Proposition~\ref{p:prox:e}\ref{i:both:conv},
\begin{equation}\label{e:pconv}
\lprox{\phi}{f}{\lambda}=(\nabla\phi+\lambda\partial f)^{-1}\circ\nabla\phi.
\end{equation}
As $(\reli\dom f)\cap U\neq\varnothing$ and
$\phi^*$ essentially smooth, we have that $(\lambda f)^*\square\phi^*=(\phi+\lambda f)^*$
is essentially smooth, see, e.g., \cite[Corollary 26.3.2]{Rock70},
so differentiable because $\dom\phi^*=\RR^n$. Then
\begin{equation}\label{e:subconj}
(\nabla\phi+\lambda\partial f)^{-1}=\nabla (\phi+\lambda f)^*.
\end{equation}
Now \cite[Theorem 16.4]{Rock70} implies
$(\phi+\lambda f)^*=\lambda(f+\phi/\lambda)^*(\cdot/\lambda)=\lambda \big(f^*\square (\phi/\lambda)^*\big)(\cdot/\lambda)$
and $\Box$ is exact.  By \cite[Proposition 16.61(i)]{BC}, for every $y\in\RR^n$,
\begin{align}\label{e:gconj}
\nabla (\phi+\lambda f)^*(y) &=\nabla \big(f^*\square (\phi/\lambda)^*\big)(y/\lambda)=\nabla (\phi/\lambda)^*\big(y/\lambda-\aprox{(\phi/\lambda)^*}{f^*}(y/\lambda)\big)\nonumber\\
&=\nabla\phi^*\big(\lambda(y/\lambda-\aprox{(\phi/\lambda)^*}{f^*}(y/\lambda))\big)
=\nabla\phi^*\big(y-\lambda\aprox{1/\lambda\timess\phi^*}{f^*}(y/\lambda)\big).
\end{align}
It follows from \eqref{e:pconv}, \eqref{e:subconj} and \eqref{e:gconj} that for $x\in\IDD$,
$$\lprox{\phi}{f}{\lambda}(x)=\nabla (\phi+\lambda f)^*(\nabla\phi(x))
=\nabla\phi^*\big(\nabla\phi(x)-\lambda\aprox{1/\lambda\timess\phi^*}{f^*}(\nabla\phi(x)/\lambda)\big),$$
as required.
\end{proof}
\begin{corollary}\label{c:moreau:p} Suppose that $f\in\GF$ and $(\reli\dom f)\cap U\neq\varnothing$.
Then for $\lambda>0$ one has
$$(\forall x\in U)\ x=\nabla\phi^* \big(\lprox{\phi^*}{f^*}{\lambda}(\nabla \phi(x))\big)+\lambda\aprox{1/\lambda\timess\phi}{f}
\big(x/\lambda\big).
$$
\end{corollary}
\begin{proof} In view of $\ran\nabla\phi=\RR^n$, Proposition~\ref{t:prox:aniso}
gives
$(\forall y\in\RR^n)\ y=\nabla\phi \big(\lprox{\phi}{f}{\lambda}(\nabla\phi^*(y))\big)+\lambda\aprox{1/\lambda\timess\phi^*}{f^*}
\big(y/\lambda\big).$
The result follows by using this identity for $f^*$ and $\phi^*$.
\end{proof}
\begin{remark} When $\lambda=1$, Corollary~\ref{c:moreau:p} recovers \cite[Theorem 1(ii)]{com13}.
\end{remark}

\section{The Bregman proximal average}\label{s:clarkemor}
Let $f_{1}, f_{2}:\RR^n\rightarrow\RX$. In the rest of the paper our standing assumptions on
$f_{1}, f_{2}$, $\alpha$ and $\lambda$ are:
\begin{itemize}
\item[\bfseries A3] Both $f_{1}$ and $ f_{2}$ are proper lower semicontinuous and
prox-bounded with thresholds $\lambda_{f_{1}}, \lambda_{f_{2}}>0$
respectively, and $\bl :=\min\{\lambda_{f_{1}},\lambda_{f_{2}}\}$.
\item[\bfseries A4]
$\dom f_{i}\cap \dom\phi\neq\varnothing$ for $i=1, 2$, $\alpha\in [0,1]$, and $\lambda\in ]0, \bl[$.
\end{itemize}
We define the \emph{$\alpha$-weighted Bregman proximal average with parameter $\lambda$} of $f_{1}, f_{2}$ with respect to the Legendre function $\phi$
 by
\begin{equation}\label{d:pave}
\averagef :=\left[\alpha\left(f_{1}+\frac{1}{\lambda}\phi\right)^*+(1-\alpha)\left(f_{2}+\frac{1}{\lambda}\phi
\right)^*\right]^*-\frac{1}{\lambda}\phi,
\end{equation}
with the convention that $+\infty-(+\infty)=+\infty$, $+\infty -r=+\infty$ for every $r\in\RR$.
As we shall see later that
$\dom \left[\alpha\left(f_{1}+\frac{1}{\lambda}\phi\right)^*+(1-\alpha)\left(f_{2}+\frac{1}{\lambda}\phi
\right)^*\right]^*\subseteq \dom\phi,$
so \eqref{d:pave} means that
\begin{equation}
\averagef(x)=\begin{cases}
\left[\alpha\left(f_{1}+\frac{1}{\lambda}\phi\right)^*+(1-\alpha)\left(f_{2}+\frac{1}{\lambda}\phi
\right)^*\right]^*(x)-\frac{1}{\lambda}\phi(x), &\text{ if $x\in\dom\phi$;}\\
+\infty, &
\text{ if $x\not\in\dom\phi$.}
\end{cases}
\end{equation}
Therefore, it is possible that $\averagef(x)=+\infty$ when $x\in\dom\phi$.


\begin{lemma}\label{l:lower:cond}
\begin{enumerate}
\item\label{i:u} The function $\averagef$ is always lower semicontinuous on $U$.
\item\label{i:dphi} If $\dom\phi$ is closed, and $\phi$ is relatively continuous on $\dom\phi$, then
$\averagef$ is lower semicontinuous on $\RR^n$. Suppose one of the following holds:
\begin{enumerate}
\item\label{i:poly1} $\dom\phi$ is polyhedral.
\item\label{i:poly2} $\dom\phi$ is locally simplicial.
\end{enumerate}
Then $\phi$ is relatively continuous on $\dom\phi$.
\end{enumerate}
\end{lemma}
\begin{proof}
\ref{i:u}: This is because that
$\phi$ is continuous on $U$ and
$\left[\alpha\left(f_{1}+\frac{1}{\lambda}\phi\right)^*+(1-\alpha)\left(f_{2}+\frac{1}{\lambda}\phi
\right)^*\right]^*$ is lower semicontinuous on $U$.

\ref{i:dphi}: On the open set $\RR^n\setminus\dom\phi$, $\averagef\equiv +\infty$, so
$\averagef$ is lower semicontinuous
on $\RR^n\setminus\dom\phi$. Now let $x_{0}\in \dom\phi$. Then
\begin{align}
& \liminf_{x\rightarrow x_{0}}\averagef(x)
=\liminf_{x\rightarrow x_{0}, x\in\dom\phi}\averagef(x)
\\
&=\liminf_{x\rightarrow x_{0},x\in\dom\phi}\left[\alpha\left(f_{1}+\frac{1}{\lambda}\phi\right)^*
+(1-\alpha)\left(f_{2}+\frac{1}{\lambda}\phi
\right)^*\right]^*(x)-\lim_{x\rightarrow x_{0},x\in\dom\phi}\frac{1}{\lambda}\phi(x)\\
&
\geq \left[\alpha\left(f_{1}+\frac{1}{\lambda}\phi\right)^*
+(1-\alpha)\left(f_{2}+\frac{1}{\lambda}\phi
\right)^*\right]^*(x_{0})-\frac{1}{\lambda}\phi(x_{0})=\averagef(x_{0}).
\end{align}
Since $x_{0}\in\dom\phi$ was arbitrary, $f$ is lower semicontinuous on $\dom\phi$.
Altogether, $\averagef$ is lower semicontinuou on $\RR^n$.
Under \ref{i:poly1} or \ref{i:poly2}, the relative continuity
of $\phi$ on $\dom\phi$ follows from \cite[Theorem 10.2]{Rock70} or
\cite[Theorem 2.35]{Rock98}.
\end{proof}

\begin{lemma}\label{l:simple}
 The following holds:
$$\frac{1}{\lambda}\left[\alpha (\lambda f_{1}+\phi)^*+(1-\alpha)(\lambda f_{2}+\phi)^*\right]^*=
\left[\alpha\left(f_{1}+\frac{1}{\lambda}\phi\right)^*+(1-\alpha)\left(f_{2}+
\frac{1}{\lambda}\phi\right)^*\right]^*.$$
\end{lemma}
\begin{proof}
Indeed, this is a simple calculation:
\begin{align}
&\frac{1}{\lambda}\left[\alpha (\lambda f_{1}+\phi)^*
+(1-\alpha)(\lambda f_{2}+\phi)^*\right]^*=\nonumber \\
& \left[\frac{1}{\lambda}\bigg(\alpha(\lambda f_{1}+\phi)^*+(1-\alpha)(\lambda f_{2}+
\phi)^*\bigg)(\lambda\cdot)\right]^*
=\left[\alpha\frac{1}{\lambda}\left(\lambda f_{1}+\phi\right)^*(\lambda\cdot)
+(1-\alpha)\frac{1}{\lambda}\left(\lambda f_{2}+
\phi\right)^*(\lambda\cdot)\right]^*\nonumber\\
& =\left[\alpha\left(\frac{\lambda f_{1}+\phi}{\lambda}\right)^*
+(1-\alpha)\left(\frac{\lambda f_{2}+\phi}{\lambda}\right)^*\right]^*
=\left[\alpha\left(f_{1}+\frac{1}{\lambda}\phi\right)^*+(1-\alpha)\left(f_{2}+
\frac{1}{\lambda}\phi\right)^*\right]^*.\nonumber
\end{align}
\end{proof}

Because of Lemma~\ref{l:lower:cond}, in the rest of the paper
our additional standing assumption on $\phi$ is:
\begin{itemize}
\item[\bfseries A5]
$\dom\phi$ is closed, $\phi$ is relatively continuous on $\dom\phi$, and
$\phi$ is twice continuously differentiable on $\IDD$ with
$\Hess \phi(u)$ being positive definite for every $u\in U$.
\end{itemize}
We are now ready for the main result of this section.

\begin{theorem}[Bregman proximal average]\label{complete}
Suppose that {\bfseries A1--A5} hold.
Then the following hold:
\begin{enumerate}
\item\label{i:p:exp}
$\averagef=\left[\alpha\timess \conv \left(f_{1}+\frac{1}{\lambda}\phi\right)\right]\Box \left[
(1-\alpha) \timess \conv \left(f_{2}+\frac{1}{\lambda}\phi\right)\right]-\frac{1}{\lambda}\phi,$
where the infimal convolution $\Box$ is exact.
\item \label{i:f:dom} $\dom \averagef=\alpha\conv(\dom f_{1}\cap\dom\phi)+(1-\alpha)\conv(\dom f_{2}\cap\dom\phi)
\subseteq\dom\phi.$
\item\label{i:lower} $\averagef$ is proper lower semicontinuous on $\RR^n$.
\item \label{i:breg:conv} $\lambda\averagef+\phi\in\GF$.
\item\label{i:prox:b}
 The function $\averagef$ is prox-bounded below with its prox-bound
 $\lambda_{f}\geq\bl$.

\item\label{i:env:sum} $\lenv{\phi}{\averagef}{\lambda}=\alpha \lenv{\phi}{f_{1}}{\lambda}+(1-\alpha) \lenv{\phi}{f_{2}}{\lambda}$.
\item\label{i:prox:sum}
 $(\forall x\in U)\ \lprox{\phi}{\averagef}{\lambda}(x)=\alpha \conv(\lprox{\phi}{f_{1}}{\lambda}(x))+(1-\alpha)\conv(\lprox{\phi}{f_{2}}{\lambda}(x))$.
 \item\label{i:prox:h} When $\alpha=0$, $\averagef=\lhul{\phi}{f_{2}}{\lambda}$;
 when $\alpha=1$, $\averagef=\lhul{\phi}{f_{1}}{\lambda}$; when $f_{1}=f_{2}=f$,
 $\averagef=\lhul{\phi}{f}{\lambda}$.
\end{enumerate}
\end{theorem}
\begin{proof}
\ref{i:p:exp}:
Since $\dom (f_{1}+1/\lambda\phi)^*=\RR^n=\dom(f_{2}+1/\lambda\phi)^*$, by \cite[Theorem 16.4]{Rock70},
\begin{equation}\label{e:another}
\averagef=\left[\alpha\left(f_{1}+\frac{1}{\lambda}\phi\right)^{**}\left(\frac{\cdot}{\alpha}\right)\right]\Box
\left[(1-\alpha)\left(f_{2}+\frac{1}{\lambda}\phi\right)^{**}\left(\frac{\cdot}{(1-\alpha)}\right)\right]
-\frac{1}{\lambda}\phi,
\end{equation}
and the infimal convolution $\Box$ is exact.
Because $f_{1}+1/\lambda\phi$ and $f_{2}+1/\lambda\phi$ are
$1$-coercive by Proposition~\ref{p:phi:hull},
\cite[Lemma 3.3]{ben96}
gives
$$\left(f_{1}+\frac{1}{\lambda}\phi\right)^{**}=\conv\left(f_{1}+\frac{1}{\lambda}\phi\right),
\quad \left(f_{2}+\frac{1}{\lambda}\phi\right)^{**}=\conv\left(f_{2}+\frac{1}{\lambda}\phi\right).$$
Hence \ref{i:p:exp} holds.

\ref{i:f:dom}: Because
$\dom \left[\conv\left(f_{i}+\frac{1}{\lambda}\phi\right)\right]=\conv(\dom f_{i} \cap\dom\phi) \text{ with $i=1, 2$},$
 by \cite[Proposition 12.6(ii)]{BC} and \ref{i:p:exp} we obtain
\begin{align}
\dom \averagef &=[\alpha \conv(\dom f_{1}\cap\dom\phi)+(1-\alpha) \conv(\dom f_{2}\cap\dom\phi)]\cap\dom\phi\\
& =\alpha \conv(\dom f_{1}\cap\dom\phi)+(1-\alpha) \conv(\dom f_{2}\cap\dom\phi),
\end{align}
where the second \enquote{=} follows from the convexity of $\dom \phi$.

\ref{i:lower}: By \ref{i:f:dom}, $\dom \averagef\neq\varnothing$; by \ref{i:p:exp},
$\averagef>-\infty$; by Lemma~\ref{l:lower:cond}\ref{i:dphi}, $\averagef$ lower semicontinuous. Therefore,
\ref{i:lower} is verified.

\ref{i:breg:conv}: By \eqref{d:pave} and \ref{i:f:dom}, we have
$$\lambda \averagef+\phi=\lambda \left[\alpha\left(f_{1}+\frac{1}{\lambda}\phi\right)^*+
(1-\alpha)\left(f_{2}+\frac{1}{\lambda}\phi
\right)^*\right]^*,$$
so $\lambda \averagef+\phi\in\GF$.

\ref{i:prox:b}:
Let $0<\lambda<\tl<\bl$. By Proposition~\ref{p:bound}, there exists
$c\in\RR$ such that
$f_{i}+\frac{1}{\tl}\phi\geq c$ for $i=1,2$. This implies
\begin{align}
f_{i}+\frac{1}{\lambda}\phi &=f_{i}+\frac{1}{\tl}\phi+\bigg(\frac{1}{\lambda}-\frac{1}{\tl}\bigg)
\phi \geq c+\bigg(\frac{1}{\lambda}-\frac{1}{\tl}\bigg)\phi,
\end{align}
so
$\bigg(f_{i}+\frac{1}{\lambda}\phi)\bigg)^{**}\geq c+\bigg(\frac{1}{\lambda}-\frac{1}{\tl}\bigg)\phi$
because $\phi\in \GF$.
In view of \eqref{e:another}, $\forall x\in\dom\phi$ we have $\averagef(x)$
\begin{align}
&\geq \left[\alpha\bigg(c+\bigg(\frac{1}{\lambda}-\frac{1}{\tl}\bigg)\phi\bigg)
\bigg(\frac{\cdot}{\alpha}\bigg)\right]\Box\left[
(1-\alpha)\bigg(c+\bigg(\frac{1}{\lambda}-\frac{1}{\tl}\bigg)\phi\bigg)
\bigg(\frac{\cdot}{1-\alpha}\bigg)\right](x)-\frac{1}{\lambda}\phi(x)\\
&=\inf_{u\in\RR^n}\bigg[c+\alpha\bigg(\frac{1}{\lambda}-\frac{1}{\tl}\bigg)\phi
\bigg(\frac{u}{\alpha}\bigg)+
(1-\alpha)\bigg(\frac{1}{\lambda}-\frac{1}{\tl}\bigg)\phi
\bigg(\frac{x-u}{1-\alpha}\bigg)\bigg]-\frac{1}{\lambda}\phi(x)\\
&=c+\bigg(\frac{1}{\lambda}-\frac{1}{\tl}\bigg)\inf_{u\in\RR^n}\bigg[\alpha\phi
\bigg(\frac{u}{\alpha}\bigg)+
(1-\alpha)\phi
\bigg(\frac{x-u}{1-\alpha}\bigg)\bigg]-\frac{1}{\lambda}\phi(x)\label{e:infc1}\\
&=c+\bigg(\frac{1}{\lambda}-\frac{1}{\tl}\bigg)\phi(x)-\frac{1}{\lambda}\phi(x)=c-\frac{1}{\tl}\phi(x),\label{e:infc2}
\end{align}
where from \eqref{e:infc1} to \eqref{e:infc2} we use the convexity of $\phi$.
Hence $\averagef+\frac{1}{\tl}\phi\geq c$ on $\dom\phi$, and so $\averagef+
\frac{1}{\tl}\phi\geq c$
on $\RR^n$.
Because $\tl\in ]0,\bl[$ was arbitrary, we conclude that $\lambda_{f}\geq \bl$
by Proposition~\ref{p:bound}.

\ref{i:env:sum}: Since $\averagef$ is proper lower semicontinuous by \ref{i:lower},
 it follows from Corollary~\ref{c:breg:convex} and
Proposition~\ref{p:prox:e} that
\begin{equation}\label{e:plus}
\lambda \averagef+\phi=(\phi^*-\lambda\lenv{\phi}{\averagef}{\lambda}\circ\nabla \phi^*)^*, \text{ and }
\end{equation}
$$\lprox{\phi}{\averagef}{\lambda} \text{ is convex-valued.}$$
Using Lemma~\ref{l:simple}, we obtain
\begin{equation}\label{e:conj:sum}
\lambda\left[\alpha\left(f_{1}+\frac{1}{\lambda}\phi\right)^*+(1-\alpha)\left(f_{2}+\frac{1}{\lambda}\phi
\right)^*\right]^*=\left[\alpha(\lambda f_{1}+\phi)^*+(1-\alpha)(\lambda f_{2}+\phi)^*\right]^*.
\end{equation}
Fact~\ref{f:env:conj} gives
\begin{equation}\label{e:f:env}
(\lambda f_{i}+\phi)^*=\phi^*-\lambda\lenv{\phi}{f_{i}}{\lambda}\circ\nabla\phi^*,
\end{equation}
which implies that $\phi^*-\lambda\lenv{\phi}{f_{i}}{\lambda}\circ\nabla\phi^*$ is convex.
Combining equations~\eqref{d:pave} and \eqref{e:plus}--\eqref{e:f:env} yields
\begin{align}\label{e:env:average}
(\phi^*-\lambda\lenv{\phi}{\averagef}{\lambda}\circ\nabla \phi^*)^*& =\left[\alpha(\phi^*-
\lambda\lenv{\phi}{f_{1}}{\lambda}\circ\nabla\phi^*)+(1-\alpha)(\phi^*-
\lambda\lenv{\phi}{f_{2}}{\lambda}\circ\nabla\phi^*)\right]^*\\
&=\left[-\alpha
\lambda\lenv{\phi}{f_{1}}{\lambda}\circ\nabla\phi^*-(1-\alpha)
\lambda\lenv{\phi}{f_{2}}{\lambda}\circ\nabla\phi^*+\phi^*\right]^*.
\end{align}
Because $\phi$ is coercive, $\phi^*$ is real-valued on $\RR^n$.
Taking conjugate both sides, followed by subtracting both sides by $\phi^*$, and using
the fact that $\nabla\phi^*$ is an isomorphism lead to
$$\lenv{\phi}{\averagef}{\lambda}=\alpha\lenv{\phi}{f_{1}}{\lambda}+(1-\alpha)
\lenv{\phi}{f_{2}}{\lambda} \text{ on $U$.}$$

\ref{i:prox:sum}: By \ref{i:env:sum},
the sum rule of Clarke subdifferential
\cite[Corollary 10.9]{Rock98} or \cite[Proposition 2.3.3, Corollary 3]{Clarke}
gives
$$\partial_{C}(-\lenv{\phi}{\averagef}{\lambda})=\alpha \partial_{C}(-\lenv{\phi}{f_{1}}{\lambda})+
(1-\alpha)\partial_{C} (-\lenv{\phi}{f_{2}}{\lambda}), $$
in which \enquote{=} holds because both $-\lenv{\phi}{f_{1}}{\lambda}$ and
$-\lenv{\phi}{f_{2}}{\lambda}$
are locally Lipschitz and Clarke regular.
Because of \ref{i:prox:b}, we can
apply Fact~\ref{f:derivative} to obtain
\begin{align}
& \frac{1}{\lambda}\nabla^{2}\phi(x)[\conv(\lprox{\phi}{\averagef}{\lambda}(x))-x]\\
&=\alpha \frac{1}{\lambda}\nabla^{2}\phi(x)[\conv(\lprox{\phi}{f_{1}}{\lambda}(x))-x]+
(1-\alpha)\frac{1}{\lambda}\nabla^{2}\phi(x)[\conv(\lprox{\phi}{f_{2}}{\lambda}(x))-x].
\end{align}
Multiplying both sides by $(\nabla^{2}\phi(x))^{-1}$ and simplifications give
$$\conv(\lprox{\phi}{\averagef}{\lambda}(x))=\alpha [\conv(\lprox{\phi}{f_{1}}{\lambda}(x))]
+(1-\alpha)[\conv(\lprox{\phi}{f_{2}}{\lambda}(x)].$$
Since $\lprox{\phi}{\averagef}{\lambda}(x)$ is convex by \ref{i:breg:conv} and
Fact~\ref{p:prox:e}\ref{i:convexf}, \ref{i:prox:sum} is proved.

\ref{i:prox:h}: Apply Proposition~\ref{p:phi:hull}\ref{i:h:c}.
\end{proof}

\begin{corollary} Suppose that
{\bfseries A1--A5} hold,
and that $f_{i}\in \GF$ with $\dom f_{i}\cap U\neq\varnothing$ for $i=1,2$. Then
for $\lambda\in ]0, +\infty[$,
\begin{equation}\label{e:prox:equiv}
\bigg(\partial\averagef+\frac{1}{\lambda}\nabla\phi\bigg)^{-1} =
\alpha\bigg(\partial f_{1}+\frac{1}{\lambda}\nabla\phi\bigg)^{-1}
+(1-\alpha)\bigg(\partial f_{2}+\frac{1}{\lambda}\nabla\phi\bigg)^{-1}.
\end{equation}
In particular,
$\forall x\in U$, $\partial\averagef(x) =\hat{\partial}\averagef(x)=$
\begin{align}\label{e:pointwise}
&
\left[\alpha\bigg(\partial f_{1}+\frac{1}{\lambda}\nabla\phi\bigg)^{-1}
+(1-\alpha)\bigg(\partial f_{2}+\frac{1}{\lambda}\nabla\phi\bigg)^{-1}\right]^{-1}(x)-\frac{1}{\lambda}\nabla\phi(x).
\end{align}
\end{corollary}
\begin{proof} By Corollary~\ref{c:proxbound}, $\bl=+\infty$. To see \eqref{e:prox:equiv},
apply Theorem~\ref{complete}\ref{i:prox:sum} and Fact~\ref{p:prox:e}\ref{i:convexf}\&\ref{i:both:conv}.
Next,
\eqref{e:pointwise} follows from \eqref{e:prox:equiv} and that $\averagef=\left(\averagef+\frac{1}{\lambda}\phi
\right)-
\frac{1}{\lambda}\phi$ being a difference of a convex function and a $C^1$ function is Clarke regular.
\end{proof}

\begin{remark}
Note that while $\partial f_{i}$ is monotone, $\partial\averagef$ may be not monotone;
see, e.g., Example~\ref{e:simple}.
\end{remark}

Let us give a special case when both $f_{1}, f_{2}$ are indicator functions of closed subsets.
This highlights the connection to averaged Bregman projections,
which solve feasibility problems.
As in \cite{bwyy}, we define Bregman nearest distance function and nearest-point map.
\begin{definition}
The \emph{left Bregman nearest-distance function} to $C$ is defined
by
\begin{equation}
\bD{C}\:  \colon U\to \RPX \colon y\mapsto
\inf_{x\in C}D_{\phi}(x,y),
\end{equation}
and the \emph{left Bregman nearest-point map} (i.e., the classical
Bregman projector) onto $C$ is
$$\bproj{C}\colon \IDD \rightrightarrows \IDD \colon y\mapsto
\underset{x\in C}{\operatorname{argmin}}\;\: D_{\phi}(x,y) = \{x\in C\colon D_{\phi}(x,y)
= \bD{C}(y)\}.$$
\end{definition}
Using Lemma~\ref{l:simple} and Fact~\ref{f:env:conj}, we can write the proximal average:
$$\averagef=\frac{1}{\lambda}[\phi^*-\alpha\lambda  \lenv{\phi}{f_{1}}{\lambda}\circ\nabla\phi^*-
(1-\alpha)\lambda\lenv{\phi}{f_{2}}{\lambda}\circ\nabla\phi^*]^*-\frac{1}{\lambda}\phi.$$
In view of $\lenv{\phi}{\iota_{C}}{\lambda}=1/\lambda \bD{C}$,
$\lprox{\phi}{\iota_{C}}{\lambda}
=\bproj{C}$, we obtain the following result.

\begin{corollary}\label{c:sets}
Suppose that {\bfseries A1--A5} hold, and that
$f_{i} :=\iota_{C_{i}}$ with $C_{i}\subseteq\RR^n$ being nonempty and closed for $i=1, 2$.
Then for $\lambda\in ]0, +\infty[$ the following hold:
\begin{enumerate}
\item\label{i:set1}
 $\averagef=\frac{1}{\lambda}[\phi^*-\alpha\bD{C_{1}}\circ\nabla\phi^*-
(1-\alpha)\bD{C_{2}}\circ\nabla\phi^*]^*-\frac{1}{\lambda}\phi.$
\item\label{i:dom:ind} $\dom \averagef=\alpha\conv(C_{1}\cap\dom\phi)+(1-\alpha)\conv(C_{2}\cap\dom\phi)
\subseteq\dom\phi.$

\item$\lenv{\phi}{\averagef}{\lambda}=\alpha \bD{C_{1}}+(1-\alpha) \bD{C_{2}}$.
\item\label{i:proj:ind}
 $(\forall x\in U)\ \lprox{\phi}{\averagef}{\lambda}(x)=\alpha
 \conv\bproj{C_{1}}(x)+(1-\alpha)\conv\bproj{C_{2}}(x)$.
\end{enumerate}
If, in addition, $C_{1}, C_{2}$ are convex, then
\begin{enumerate}[(a)]
\item \label{i:set:exe}
$\averagef(x)=$
\begin{align}
&\frac{1}{\lambda}\inf\{\alpha D_{\phi}(y_{1},x)+(1-\alpha)D_{\phi}(y_{2},x):
\ y_{i}\in C_{i}\cap \dom\phi, i=1, 2,
\alpha y_{1}+(1-\alpha)y_{2}=x\}, \nonumber \text{ and }
\end{align}
\item the
\enquote{conv} operations in \ref{i:dom:ind}
and \ref{i:proj:ind} are superfluous.
\end{enumerate}
\end{corollary}
\begin{proof} \ref{i:set1}-\ref{i:proj:ind} follow from Theorem~\ref{complete}.
To see \ref{i:set:exe},
we consider
\begin{align}
&\averagef(x)=
\bigg[\alpha\timess\bigg(\iota_{C_{1}}+\frac{1}{\lambda}\phi\bigg)\bigg]
\square \bigg[(1-\alpha)\timess \bigg(\iota_{C_{2}}+\frac{1}{\lambda}\phi\bigg)\bigg](x)
-\frac{1}{\lambda}\phi(x)\\
&=\inf_{x_{1}+x_{2}=x}\bigg(\iota_{C_{1}}(x_{1}/\alpha)
+\alpha\frac{1}{\lambda}\phi(x_{1}/\alpha)+\iota_{C_{2}}(x_{2}/(1-\alpha))
+(1-\alpha)\frac{1}{\lambda}\phi(x_{2}/(1-\alpha))\bigg)-\frac{1}{\lambda}\phi(x)\\
&=\frac{1}{\lambda}\inf\{\alpha\phi(y_{1})+(1-\alpha)\phi(y_{2})-\phi(x):\
y_{i}\in C_{i}\cap\dom\phi, i=1, 2, \alpha y_{1}+(1-\alpha)y_{2}=x\}.
\end{align}
The proof is complete by using that
when $\alpha y_{1}+(1-\alpha)y_{2}=x$, one has
\begin{align}
& \alpha\phi(y_{1})+(1-\alpha)\phi(y_{2})-\phi(x)\\
&=\alpha (\phi(y_{1})-\phi(x)-\scal{\nabla\phi(x)}{y_{1}-x}) +
(1-\alpha)(\phi(y_{2})-\phi(x)-\scal{\nabla\phi(x)}{y_{2}-x})\\
&=\alpha D_{\phi}(y_{1},x)+(1-\alpha)D_{\phi}(y_{2},x).
\end{align}
\end{proof}

\section{When is the Bregman proximal average convex?}\label{s:whenconvex}

We shall need a Bregman version of the Baillon-Haddad theorem, see, e.g., \cite{haddad, BC}. To this end, we introduce $\nphi$-firmly nonexpansive mappings.
Define the symmetrized Bregman distance $\sphi: U\times U\rightarrow\RR$ by
$\sphi(x,y)=D_{\phi}(x,y)+D_{\phi}(y,x)=\scal{\nabla \phi(x)-\nabla\phi(y)}{x-y}.$

\begin{definition}
Let $T:U\subseteq\RR^n\rightarrow U$. We say that $T$ is $\nabla\phi$-firmly nonexpanive on $U$
if
$$(\forall u\in U)(\forall v\in U)\ \scal{u-v}{Tu-Tv}\geq\scal{\nabla\phi(Tu)-\nabla\phi(Tv)}{Tu-Tv}
=\sphi(Tu, Tv).$$
\end{definition}
When $\phi(x)=1/2\|x\|^2$, a $\nabla\phi$-firmly nonexpansive mapping is the usual firmly nonexpansive mapping;
see, e.g., \cite[Proposition 4.4]{BC}.

\begin{lemma}\label{l:strong}
Suppose that $g\in\GF$, $\dom g\subseteq\dom\phi$,
 and $(\reli\dom g)\cap U \neq\varnothing$.
Then the following are equivalent:
\begin{enumerate}
\item\label{i:g} $g:\RR^{n}\rightarrow\RX$ is $\phi$-strongly convex, i.e., $g=f+\phi$
for a convex function $f\in\GF$.
\item\label{i:gdual} $g^*$ is a $\phi^*$-anisotropic envelope of $f^*$ with
$f\in\GF$, i.e., $g^*=f^*\square \phi^*$.
\item\label{i:gs} $g^*$ is differentiable with $\nabla g^*$ being
$\nphi$-firmly nonexpansive on $\RR^n$.
\item\label{i:diff} $(\phi^*-g^*)\circ\nabla\phi=\lambda \lenv{\phi}{f}{\lambda}$ for a convex function
$f\in \GF$ and $\lambda>0$.
\item\label{i:gs:prox}
 $g^*$ is differentiable on $\RR^n$ with $\nabla g^*\circ \nabla\phi=\lprox{\phi}{f}{1}$ for
 some $f\in \GF$.
\end{enumerate}
\end{lemma}
\begin{proof}
\ref{i:g}$\Rightarrow$\ref{i:gdual}: Since $\varnothing\neq \reli\dom g=
\reli[(\dom f)\cap (\dom \phi)]=(\reli\dom f) \cap (\reli\dom\phi)
\subseteq (\dom f) \cap U$, we have $\dom f\cap \inte\dom\phi\neq \varnothing$.
Apply the Attouch-Brezis theorem \cite[Theorem 15.3]{BC}.
\ref{i:gdual}$\Rightarrow$\ref{i:g}: Take the conjugation both sides to obtain
$g=g^{**}=f^{**}+\phi^{**}=f+\phi$; see, e.g., \cite[Theorem 13.37]{BC}.

\ref{i:g}$\Rightarrow$\ref{i:gs}: Since $\phi$ is $1$-coercive, so is $g$ and hence
 $\ran \partial g=\RR^n$.
Because
$\dom g\cap\inte\dom\phi\neq\varnothing$ implies $\dom f\cap\inte\dom\phi\neq\varnothing$,
we have $\partial g=\partial f+\partial\phi$, so $\dom\partial g\subseteq\dom\partial\phi$.
As $f$ is convex, $\phi$ is essentially strictly
convex, we see that $g$ is essentially strictly convex, so
$g^*$ is essentially smooth.
Using $u\in\partial g(x), v\in\partial g(y)$ if and only if
$x=\nabla g^*(u), y=\nabla g^*(v)$, we obtain
\begin{align}
&\scal{\partial g(x)-\partial g(y)}{x-y}\geq \scal{\nphi(x)-\nphi(y)}{x-y}\\
&\Leftrightarrow \scal{u-v}{\nabla g^*(u)-\nabla g^*(v)}
\geq \scal{\nphi(\nabla g^*(u))-\nphi(\nabla g^*(v))}{\nabla g^*(u)-\nabla g^*(v)}
\end{align}
for all $u,v\in\RR^n$.

\ref{i:gs}$\Rightarrow$\ref{i:g}: Since
\begin{align}
&(\forall u, v\in\RR^n)\ \scal{u-v}{\nabla g^*(u)-\nabla g^*(v)}
\geq \scal{\nphi(\nabla g^*(u))-\nphi(\nabla g^*(v))}{\nabla g^*(u)-\nabla g^*(v)}\\
& \Leftrightarrow (\forall x,y\in\dom\partial g \cap U)\
\scal{\partial g(x)-\partial g(y)}{x-y}\geq \scal{\nphi(x)-\nphi(y)}{x-y},
\end{align}
the function $g-\phi$ is convex on convex subsets of
$(\dom\partial g)\cap U\supseteq (\reli\dom g)\cap U=\reli(\dom g\cap \dom\phi)=\reli \dom g$.
Define $\tilde{f}(x)=g(x)-\phi(x)$ if $x\in\reli \dom g$,
and $+\infty$ otherwise. Since $\tilde{f}$ is proper and convex, by \cite[Theorem 2.35]{Rock98},
the lower semicontinuous hull
$f=\closu \tilde{f}$ is proper, so $f\in\GF$. We claim that
$g=f+\phi$ on $\dom g$. Indeed, as $g-\phi$ is relatively continuous on $\reli\dom g$,
$f=\closu(g-\phi)=g-\phi$, which gives $g=f+\phi$ on $\reli\dom g$.
Take $x_{0}\in \reli \dom g\cap U$, which is possible by the assumption, and let
$x\in \dom g$. Then, by \cite[Theorem 2.36]{Rock98},
$$f(x)=\lim_{\tau\uparrow 1}f((1-\tau)x_{0}+\tau x)
=\lim_{\tau\uparrow 1}(g((1-\tau)x_{0}+\tau x)
-\phi((1-\tau)x_{0}+\tau x))
=g(x)-\phi(x)$$
because both $g, \phi\in\GF$. Therefore, $f=g-\phi$ on $\dom g$. As $\dom g\subset\dom\phi$,
we get $g=f+\phi$ on $\dom g$ and $f\in\GF$. However, at this stage, we do not know whether $
g=f+\phi$ on $\RR^n\setminus \dom g$. Now write
$g=(f+\iota_{\dom g})+\phi$. Becuase $\dom(f+\iota_{\dom g})=\dom g$, $\reli\dom g\cap\ U\neq\varnothing$
and both $(f+\iota_{\dom g})$ and $\phi$ are proper convex,
\cite[Theorem 9.3]{Rock70} gives
$$g=\closu g=\closu(f+\iota_{\dom g})+\closu\phi=\closu(f+\iota_{\dom g})+\phi$$
and $\closu(f+\iota_{\dom g})\in\GF$. This proves \ref{i:g}.

\ref{i:diff}$\Leftrightarrow$\ref{i:g}: We have
\begin{align}
 {\rm \ref{i:diff}} & \Leftrightarrow (\phi^*-g^*)\circ\nabla\phi
=\lambda \lenv{\phi}{f}{\lambda}
\Leftrightarrow \phi^*-g^*=\lambda\lenv{\phi}{f}{\lambda}\circ\nabla\phi^*\\
& \Leftrightarrow \phi^*-\lambda\lenv{\phi}{f}{\lambda}\circ\nabla\phi^*=g^*
 \Leftrightarrow (\lambda f+\phi)^*=g^* (\mbox{Fact~\ref{f:env:conj}})
\Leftrightarrow  g=\lambda f+\phi,
\end{align}
and $\lambda f\in\GF$.

\ref{i:gdual}$\Rightarrow$\ref{i:gs:prox}: \ref{i:gdual} gives $\dom g^*=\RR^n$ and
$(\forall x^*\in\RR^n)\ \nabla g^*(x^*)=\nabla\phi^*(x^*-\aprox{\phi^*}{f^*}(x^*))$. Put $x^*=\nabla
\phi(x)$ for $x\in U$ to obtain
$$\nabla g^*(\nabla\phi(x))=\nabla\phi^*(\nabla\phi(x)-\aprox{\phi^*}{f^*}(\nabla\phi(x)))=
\lprox{\phi}{f}{1}(x)$$
by Proposition~\ref{t:prox:aniso}.

\ref{i:gs:prox}$\Rightarrow$\ref{i:gdual}: \ref{i:gs:prox} gives  $(\forall x\in U)\ \nabla g^*(\nabla\phi(x))=
\lprox{\phi}{f}{1}(x)=\nabla\phi^*(\nabla\phi(x)-\aprox{\phi^*}{f^*}(\nabla\phi(x))).$
In view of $\ran\nabla\phi=\RR^n$, replacing $\nabla \phi(x)$ by $x^*$ gives
$$(\forall x^*\in\RR^n)\ \nabla g^*(x^*)=\nabla\phi^*(x^*-\aprox{\phi^*}{f^*}(x^*)=\nabla (f^*\square\phi^*)(x^*),$$
which implies $g^*=(f^*\square\phi^*)+c=(f^*+c)\square\phi^*$ for a constant $c\in\RR$.
Hence \ref{i:gdual} holds.
\end{proof}
\begin{remark} The above is an extended version of
Baillon-Haddad Theorem; see \cite[Theorem 18.15, Corollary 18.17]{BC}, \cite{haddad}.
$\phi$-strongly convex functions have been used in \cite{bbctw} for studying Bregman gradient algorithms.
\end{remark}

\begin{lemma}\label{l:convex:c} Let $\sphi$ be convex. Suppose that
$T_{1}, T_{2}$ are $\nphi$-firmly nonexpansive on $U$. Then
$\alpha T_{1}+(1-\alpha)T_{2}$ is $\nphi$-firmly nonexpansive on $U$.
\end{lemma}

\begin{proof}
This follows from the following calculations: $\forall u, v\in U$,
\begin{align*}
& \scal{\nphi(\alpha T_{1}u+(1-\alpha)T_{2}u)-\nphi(\alpha T_{1}v+(1-\alpha)T_{2}v)}{
(\alpha T_{1}u+(1-\alpha)T_{2}u)-(\alpha T_{1}v+(1-\alpha)T_{2}v)}\\
&=\sphi(\alpha T_{1}u+(1-\alpha)T_{2}u,\alpha T_{1}v+(1-\alpha)T_{2}v)
=\sphi(\alpha(T_{1}u,T_{1}v)+(1-\alpha)(T_{2}u,T_{2}v))\\
&\leq \alpha \sphi(T_{1}u, T_{1}v)+(1-\alpha)\sphi(T_{2}u, T_{2}v)\ \ (\text{$\sphi$ being convex}) \\
&\leq \alpha\scal{u-v}{T_{1}u-T_{1}v}+(1-\alpha)\scal{u-v}{T_{2}u-T_{2}v}\ \ (\text{$T_{i}$ being $\nphi$-firmly
nonexpansive})\\
&=\scal{u-v}{\alpha T_{1}u+(1-\alpha)T_{2}u-(\alpha T_{1}v+(1-\alpha)T_{2}v)}.
\end{align*}
\end{proof}

Here is the main result of this section.
\begin{theorem}[convexity of Bregman proximal average]\label{t:convex}
Let {\bfseries A1--A5} hold, and let $\sphi$ be convex.
Suppose that $f_{i}\in \GF$ and $(\reli\dom f_{i})\cap U\neq\varnothing$
for $i=1, 2$.
Then $\averagef$ is convex.
\end{theorem}
\begin{proof}
Recall that
\begin{equation}
\averagef=\left[\alpha\left(f_{1}+\frac{1}{\lambda}\phi\right)^*+(1-\alpha)\left(f_{2}+\frac{1}{\lambda}\phi
\right)^*\right]^*-\frac{1}{\lambda}\phi.
\end{equation}
Since $f_{i}+\frac{1}{\lambda}\phi$ is $\phi/\lambda$-strongly convex, by
Lemma~\ref{l:strong}\ref{i:gs}, each $T_{i}=\nabla \left(f_{i}+\frac{1}{\lambda}\phi\right)^*$ is
$\nabla\phi/\lambda$-firmly nonexpansive. Lemma~\ref{l:convex:c} implies
$\alpha\nabla\left(f_{1}+\frac{1}{\lambda}\phi\right)^*+(1-\alpha)\nabla
\left(f_{2}+\frac{1}{\lambda}\phi\right)^*
$
is $\nabla\phi/\lambda$-firmly nonexpansive. Because
$$\dom \left[\alpha\left(f_{1}+\frac{1}{\lambda}\phi\right)^*+(1-\alpha)\left(f_{2}+\frac{1}{\lambda}\phi
\right)^*\right]^*=\alpha (\dom f_{1}\cap\dom\phi)+(1-\alpha) (\dom f_{2}\cap\dom \phi),$$
by the assumption, we have
$\reli[\alpha (\dom f_{1}\cap\dom\phi)+(1-\alpha) (\dom f_{2}\cap\dom \phi)]
\cap U\neq\varnothing.$
Apply Lemma~\ref{l:strong}\ref{i:gs} again to
obtain that
$$\left[\alpha\left(f_{1}+\frac{1}{\lambda}\phi\right)^*+(1-\alpha)\left(f_{2}+\frac{1}{\lambda}\phi
\right)^*\right]^*$$
is $\phi/\lambda$-strongly convex. Hence $\averagef$ is convex by Lemma~\ref{l:strong}\ref{i:g}.
\end{proof}
\begin{remark} Clearly, the joint convexity of $D_{\phi}$ implies the convexity of $S_{\phi}$.
 For conditions on joint convexity of $D_{\phi}$, see \cite{bauschkebor}.
\end{remark}

\begin{corollary}
 Let {\bfseries A1--A5} hold, and let $\sphi$ be convex.
Suppose that $f_{i}\in\GF$ and $(\reli\dom f_{i})\cap U\neq\varnothing$
for $i=1, 2$.  Then $\averagef$ is convex, and
$(\forall x\in U)\ \lprox{\phi}{\averagef}{\lambda}(x)=\alpha \lprox{\phi}{f_{1}}{\lambda}(x)+(1-\alpha)\lprox{\phi}{f_{2}}{\lambda}(x).$
\end{corollary}
\begin{proof}
Apply Theorem~\ref{t:convex}\ref{i:prox:sum} and Proposition~\ref{p:prox:e}\ref{i:both:conv}.
\end{proof}

The example below illustrates that Theorem~\ref{t:convex} fails without the convexity of
$S_{\phi}$.

\begin{example}\label{e:simple}
For $\phi(x)=|x|^3$, simple calculus shows that
$S_{\phi}(x,y)=(3|x|x-3|y|y)(x-y)$ is not convex on $\RP^2$.
Let $\lambda=1$, and let $a>0$, $f_{1} :=\iota_{\{a\}}, f_{2}:\equiv 0$ on $\RR$.
Then
\begin{equation}\label{e:cubic}
\averagef(x)=\alpha |a|^3+ \frac{|x-\alpha a|^3}{(1-\alpha)^2}-|x|^3,
\end{equation}
and $\averagef$ is not convex.
\end{example}
\begin{proof}
Because $f_{1}, f_{2}\in \GF$ and Theorem~\ref{complete}\ref{i:p:exp},
we have
\begin{equation}\label{e:simple:c}
\averagef=\left[\alpha\left(f_{1}+\frac{1}{\lambda}\phi\right)\left(\frac{\cdot}{\alpha}\right)\right]\Box
\left[(1-\alpha)\left(f_{2}+\frac{1}{\lambda}\phi\right)\left(\frac{\cdot}{(1-\alpha)}\right)\right]
-\frac{1}{\lambda}\phi.
\end{equation}
As
$\alpha\left(f_{1}+\phi\right)\left(\frac{\cdot}{\alpha}\right)=\iota_{\{\alpha a\}}+\alpha\phi(a)$ and
$(1-\alpha)\left(f_{2}+\phi\right)\left(\frac{\cdot}{1-\alpha}\right)=
(1-\alpha)\phi\left(\frac{\cdot}{1-\alpha}\right),$
by \eqref{e:simple:c} we have
\begin{align}
\averagef(x) &=\inf_{y}\left\{\iota_{\{\alpha a\}}(y)+\alpha\phi(a)+
(1-\alpha)\phi\left(\frac{x-y}{1-\alpha}\right)\right\}
-\phi(x)\\
&=\alpha\phi(a)+(1-\alpha)\phi\left(\frac{x-\alpha a}{1-\alpha}\right)-\phi(x).
\label{e:closedf}
\end{align}
Equations \eqref{e:cubic} is immediate from \eqref{e:closedf}.

When $x\geq \alpha a$,
$\averagef(x)=\frac{(x-\alpha a)^3}{(1-\alpha)^2}-x^3,
\text{ so } \averagef''(x)=\frac{6(x-\alpha a)}{(1-\alpha)^2}-6x.$
As $x\rightarrow\alpha a$, $\averagef''(x)<0$, so $\averagef$ is not convex.
\end{proof}


It is naturally to ask: If $\averagef$ is convex for all $f_{1}, f_{2}\in\GF$ and $\alpha\in]0,1[$,
what can we say about the Legendre function $\phi$ or $D_{\phi}$? This is partially answered
by the following result on $\RR$.
\begin{proposition} Let {\bfseries A1--A5} hold.
Suppose that $\averagef$ is convex for every $\alpha\in ]0,1[$, $f_{1}, f_{2}\in \GFO$. Then
$D_{\phi}$ is separably convex on $\RR^2$.
\end{proposition}
\begin{proof} Note that
\begin{equation}\label{e:p:formula}
\averagef=\left[\alpha\left(f_{1}+\frac{1}{\lambda}\phi\right)\left(\frac{\cdot}{\alpha}\right)\right]\Box
\left[(1-\alpha)\left(f_{2}+\frac{1}{\lambda}\phi\right)\left(\frac{\cdot}{(1-\alpha)}\right)\right]
-\frac{1}{\lambda}\phi.
\end{equation}
Let $f_{1}=\iota_{\{p\}}$ where $p\in \dom\phi$, and $f_{2}\equiv 0$. \eqref{e:p:formula}
gives
$$(\forall y\in U)\ \averagef(\alpha p+(1-\alpha)y)
=\frac{1}{\lambda}\bigg(\alpha\phi(p)+(1-\alpha)\phi(y)-\phi(\alpha p+(1-\alpha)y)\bigg).$$
Put $g(y)=\alpha\phi(p)+(1-\alpha)\phi(y)-\phi(\alpha p+(1-\alpha)y)$. By the assumption,
$g$ is convex for every $\alpha\in ]0,1[$, so
$(\forall y\in U)\ g''(y)=(1-\alpha)\phi''(y)-(1-\alpha)^2\phi''(\alpha p+(1-\alpha)y)\geq 0.$
This implies $\phi''(y)\geq (1-\alpha)\phi''(\alpha p+(1-\alpha)y)$, from which
$$\phi''(y)-(1-\alpha) \phi''(y)\geq (1-\alpha)[\phi''(\alpha p+(1-\alpha)y)-\phi''(y)],$$
$$\phi''(y)\geq (1-\alpha)\frac{\phi''(y+\alpha (p-y))-\phi''(y)}{\alpha}.$$
When $\alpha\downarrow 0$, we obtain $\phi''(y)\geq \phi'''(y)(p-y),$ whence
$D_{\phi}$ is separably convex by \cite[Theorem 3.3(ii)]{bauschkebor}.
\end{proof}

\section{Duality via Combettes and Reyes' anisotropic envelope and proximity operator}\label{s:duality}
The Combettes-Reyes
anisotropic envelope and proximity operator are essential
in the study of the Fenchel conjugate of Bregman proximal averages.

%
%

\begin{theorem}[Duality of Bregman proximal average] Let {\bfseries A1--A5} hold, and
let $f_{i}\in\GF$ for $i=1,2$. Then the following hold:
\begin{enumerate}
\item\label{i:primal} Suppose that $(\forall i)\ (\reli\dom f_{i})\cap U\neq\varnothing$,
and that $D_{\phi}$ is jointly convex. Then
 the anisotropic envelope and proximal mapping
of $\averagef^*$ satisfy
\begin{equation}\label{e:primal1}
\averagef^*\square (1/\lambda\timess\phi^*)=
\alpha f_{1}^*\square(1/\lambda \timess\phi^*)+(1-\alpha) f_{2}^*\square
(1/\lambda \timess\phi^*),
\end{equation}
and $\forall x^*\in \RR^n$,
\begin{align}\label{e:primal2}
& \nabla \phi^*\left(\lambda(x^*-\aprox{1/\lambda\timess\phi^*}{\averagef^*}(x^*))\right)\\
&=
\alpha \nabla \phi^*\left(\lambda(x^*-\aprox{1/\lambda\timess\phi^*}{f_{1}^*}(x^*))\right)
+(1-\alpha)\nabla \phi^*\left(\lambda(x^*-\aprox{1/\lambda\timess\phi^*}{f_{2}^*}(x^*))\right).
\nonumber
\end{align}
\item\label{i:dual} Suppose that $D_{\phi^*}$ is jointly convex. Then
 the anisotropic envelope and proximal mapping
of $\averagefd^*$ satisfy
\begin{equation}\label{e:dual1}
\averagefd^*\square (\lambda\timess\phi)=
\alpha f_{1}\square(\lambda \timess\phi)+(1-\alpha) f_{2}\square
(\lambda \timess\phi),
\end{equation}
and $\forall x\in [\alpha(\dom f_{1}^*)+(1-\alpha)(\dom f_{2}^*)+\lambda\IDD]$,
\begin{align}\label{e:dual2}
& \nabla\phi\left((x-\aprox{\lambda\timess\phi}{\averagefd}(x))/\lambda\right)\\
& =
\alpha \nabla \phi\left((x-\aprox{\lambda\timess\phi}{f_{1}}(x))/\lambda\right)
+(1-\alpha)\nabla \phi\left((x-\aprox{\lambda\timess\phi}{f_{2}}(x))/\lambda)\right).\nonumber
\end{align}
\end{enumerate}
\end{theorem}
\begin{proof}
\ref{i:primal}: By Fact~\ref{f:env:conj},
$\phi^*=(\lambda f_{i}+\phi)^*+\lambda\lenv{\phi}{f_{i}}{\lambda}\circ\phi^*$. Multiplying
both sides by $\alpha$ for $i=1$, and $(1-\alpha)$ for $i=2$, followed by adding both equations,
we have
$$\phi^*-\lambda(\alpha\lenv{\phi}{f_{1}}{\lambda}\circ\phi^*
+(1-\alpha)\lenv{\phi}{f_{2}}{\lambda}\circ\phi^*)=\alpha(\lambda f_{1}+\phi)^*
+(1-\alpha)(\lambda f_{2}+\phi)^*.$$
Theorem~\ref{complete}\ref{i:env:sum} gives
$\phi^*-\lambda\lenv{\phi}{\averagef}{\lambda}\circ\phi^*=\alpha(\lambda f_{1}+\phi)^*
+(1-\alpha)(\lambda f_{2}+\phi)^*.$
Use Fact~\ref{f:env:conj} again to obtain
\begin{equation}\label{e:conj}
(\lambda\averagef+\phi)^*=\alpha(\lambda f_{1}+\phi)^*
+(1-\alpha)(\lambda f_{2}+\phi)^*.
\end{equation}
Since $(\reli\dom f_{i})\cap U\neq \varnothing$ for $i=1,2$, by \cite[Theorem 16.4]{Rock70}
we can write
\begin{equation}\label{e:conj:exact1}
(\lambda f_{i}+\phi)^*=\lambda\timess(f_{i}^*\square (1/\lambda\timess\phi^*)),
\end{equation}
where the $\square$ is exact.
Moreover,
as $\dom\averagef=\alpha \dom f_{1}\cap\dom\phi+(1-\alpha)\dom f_{2}\cap\dom\phi$
by Theorem~\ref{complete}\ref{i:f:dom},
in view of  \cite[Theorems 6.5, 6.6]{Rock70} we have
\begin{align}\label{e:domain}
\reli \dom\averagef &=\alpha \reli(\dom f_{1}\cap\dom\phi)+(1-\alpha)
\reli(\dom f_{2}\cap\dom\phi)\\
&=\alpha (\reli\dom f_{1})\cap U+(1-\alpha)(\reli\dom f_{2})\cap U
\subseteq U.\nonumber
\end{align}
Because $D_{\phi}$ is jointly convex, $\averagef$ is convex by Theorem~\ref{t:convex}.
In view of \eqref{e:domain}, it follows from \cite[Theorem 16.4]{Rock70} that
\begin{equation}\label{e:conj:exact2}
(\lambda \averagef+\phi)^* =\lambda\timess(\averagef^*\square (1/\lambda\timess\phi^*)),
\end{equation}
and $\square$ is exact. Combining~\eqref{e:conj}, \eqref{e:conj:exact1}, and
\eqref{e:conj:exact2} gives \eqref{e:primal1}.

Since $\phi^*$ is differentiable, \cite[Proposition 16.61(i)]{BC} or \cite[Lemma 2.1]{penot}
gives
\begin{align}
\nabla [f_{i}^*\square(1/\lambda \timess\phi^*)](x^*)& = \nabla(1/\lambda\timess \phi^*)\left(x^*-\aprox{1/\lambda\timess\phi^*}{f_{i}^*}(x^*)\right)\\
& =\nabla\phi^*\left(\lambda(x^*-\aprox{1/\lambda\timess\phi^*}{f_{i}^*}(x^*))\right), \text{ and }
\end{align}
\begin{align}
\nabla [\averagef^*\square(1/\lambda \timess\phi^*)](x^*)& =\nabla(1/\lambda\timess \phi^*)\left(x^*-\aprox{1/\lambda\timess\phi^*}{\averagef^*}(x^*)\right)\\
&= \nabla \phi^*\left(\lambda(x^*-\aprox{1/\lambda\timess\phi^*}{\averagef^*}(x^*))
\right).
\end{align}
Hence,
\eqref{e:primal2} follows from \eqref{e:primal1} by taking derivatives
both sides.

\ref{i:dual}:
Note that $\dom\phi^*=\RR^n$.
Apply \ref{i:primal} with $f_{i}$ replaced by
$f_{i}^*$, $\phi$ by $\phi^*$ and $\lambda$ by $1/\lambda$, followed by using
Theorem~\ref{complete}\ref{i:f:dom} and
Proposition~\ref{p:aniso:env}\ref{i:aniso:d}.
\end{proof}

\begin{remark}(1). $D_{\phi}$ jointly convex does not mean $D_{\phi^*}$ jointly convex. For example,
for $\phi(x)=x\ln x-x$ if $x\geq 0$ and $+\infty$ otherwise, and $\phi^*(x)=\exp(x)$,
$D_{\phi}$ is jointly convex, but $D_{\phi^*}$ is not. (2). In general,
$\averagefd^*\neq \averagef$ because the latter might not be convex.
While the anisotropic envelope of $\averagefd^*$ is the convex combination
of anisotropic envelopes of $f_{i}$'s,  the Bregman envelope of $\averagef$ is the convex combination
of Bregman envelopes of $f_{i}$'s.
\end{remark}
\begin{remark}
Note that $(\forall f\in\GF)(\forall x^*\in\RR^n)\
\nabla\phi^*(x^*-\aprox{\phi^*}{f^*}(x^*))=\lprox{\phi}{f}{1}(\nabla\phi^*(x^*))$
by Proposition~\ref{t:prox:aniso}.
Thus, \eqref{e:primal2} is essentially an identity for proximal mappings, and
the same can be said for \eqref{e:dual2}.
\end{remark}

\section{Epi-continuity}\label{s:epiconv}
This section is devoted to the epi-convergence behaviors of $\averagef$ when parameters
$\lambda$ and $\alpha$ vary.
\begin{definition}
A sequence of functions $(f_{k})_{k\in \NN}$ from $\RR^n\rightarrow\RX$
epi-converges to $f$ at a point $x\in\RR^n$ if both of the following conditions are satisfied:
\begin{enumerate}
\item whenever $(x_k)_{k\in\NN}$ converges to $x$, we have
$f(x)\leq\liminf_{k\rightarrow\infty} f_{k}(x_{k})$;
\item there exists a sequence $(x_{k})_{k\in\NN}$ converges to $x$ with
$f(x)=\lim_{k\rightarrow\infty}f_{k}(x_{k})$.
\end{enumerate}
If $(f_{k})_{k\in \NN}$
epi-converges to $f$ at every $x\in C\subseteq\RR^n$, we say
$(f_{k})_{k\in \NN}$
epi-converges to $f$ on $C$. In the case of $C=\RR^n$, the functions $f_{k}$ are
 said to epi-converge to $f$, denoted by $f_{k}\epi f$.
\end{definition}
See \cite[pages 241-243]{Rock98} or \cite[page 159]{beer} for further details on epi-convergence.

\begin{theorem}[epi-continuity I of Bregman proximal average]
Let {\bfseries A1--A5} hold.
Then the following hold:
\begin{enumerate}
\item\label{i:0} As $\alpha\downarrow 0$, $\averagef\epi\lhul{\phi}{f_{2}}{\lambda}$ on $U$.
\item\label{i:1} As $\alpha\uparrow 1$, $\averagef\epi\lhul{\phi}{f_{1}}{\lambda}$ on $U$.
\end{enumerate}
In particular, when $f_{1}, f_{2}\in\GF$, we have
\begin{enumerate}[(a)]
\item\label{i:c:0} As $\alpha\downarrow 0$, $\averagef\epi f_{2}$ on $U$.
\item\label{i:c:1} As $\alpha\uparrow 1$, $\averagef\epi f_{1}$ on $U$.
\end{enumerate}
\end{theorem}
\begin{proof}
\ref{i:0}: By Proposition~\ref{p:bound}\ref{i:boundl},
each $f_{i}+\frac{1}{\lambda}\phi$ is $1$-coercive so that
its Fenchel conjugate
$\left(f_{i}+\frac{1}{\lambda}\phi\right)^*$
has a full domain. When $\alpha\downarrow 0$,
$$\left[\alpha\left(f_{1}+\frac{1}{\lambda}\phi\right)^*+(1-\alpha)\left(f_{2}+\frac{1}{\lambda}\phi
\right)^*\right]\rightarrow \left(f_{2}+\frac{1}{\lambda}\phi
\right)^*$$
pointwise, so epi-converges by \cite[Theorem 7.17]{Rock98}. By \cite[Theorem 11.34]{Rock98},
$$\left[\alpha\left(f_{1}+\frac{1}{\lambda}\phi\right)^*+(1-\alpha)\left(f_{2}+\frac{1}{\lambda}\phi
\right)^*\right]^*$$
epi-converges to $\left(f_{2}+\frac{1}{\lambda}\phi
\right)^{**}$ on $\RR^n$, so epi-converges at every point of $U$. Since $\phi$ is continuous on $U$, in view of \cite[Exercise 7.8]{Rock98},
$\left[\alpha\left(f_{1}+\frac{1}{\lambda}\phi\right)^*+(1-\alpha)\left(f_{2}+\frac{1}{\lambda}\phi
\right)^*\right]^*-\frac{1}{\lambda}\phi$
epi-converges to
$\left(f_{2}+\frac{1}{\lambda}\phi
\right)^{**}-\frac{1}{\lambda}\phi \text{ on $U$,}
$
when $\alpha\downarrow 0$.

\ref{i:1}: The proof is analogous to that of \ref{i:0}. Finally, \ref{i:c:0}\&\ref{i:c:1} hold because
Proposition\ref{p:phi:hull}\ref{i:h:c} implies
$\lhul{\phi}{f_{i}}{\lambda}=f_{i}$
on $\IDD$ when $f_{i}\in \GF$.
\end{proof}

The next result shows that the Bregman proximal average lies between
the epi-average of convexified individual functions and the arithmetic
average of individual functions.
\begin{theorem}\label{t:squeeze} Let {\bfseries A1--A5} hold.
Then the following hold:
\begin{enumerate}
\item\label{i:infimal}
$\averagef\geq \left[\alpha \conv f_{1}\left(\frac{\cdot}{\alpha}\right)\right]\Box \left[
(1-\alpha) \conv f_{2}\left(\frac{\cdot}{1-\alpha}\right)\right].$
\item\label{i:arithm}
$\averagef\leq \alpha f_{1}+(1-\alpha)f_{2} \text{ on $\dom\phi$.} $ In particular,
$\averagef\leq \alpha f_{1}+(1-\alpha)f_{2} $ if $\dom f_{1} \cap \dom f_{2}\subseteq\dom\phi$.
\end{enumerate}
\end{theorem}
\begin{proof}
\ref{i:infimal}:
Because $\phi$ is convex, we have
$\left[\alpha\phi\left(\frac{\cdot}{\alpha}\right)\right]\Box \left[
(1-\alpha)\phi\left(\frac{\cdot}{1-\alpha}\right)\right]
=\phi$ and $\conv\phi =\phi$.
It follows from Theorem~\ref{complete}\ref{i:p:exp} that $\averagef$
\begin{align*}
 &=\left[\alpha \conv \left(f_{1}+\frac{1}{\lambda}\phi\right)\left(\frac{\cdot}{\alpha}\right)\right]\Box \left[
(1-\alpha) \conv \left(f_{2}+\frac{1}{\lambda}\phi\right)\left(\frac{\cdot}{1-\alpha}\right)\right]-\frac{1}{\lambda}\phi
\\
&\geq \left[\alpha \left(\conv f_{1}+\frac{1}{\lambda}\phi\right)\left(\frac{\cdot}{\alpha}\right)\right]\Box \left[
(1-\alpha)\left(\conv f_{2}+\frac{1}{\lambda}\phi\right)\left(\frac{\cdot}{1-\alpha}\right)\right]-\frac{1}{\lambda}\phi
\\
&\geq \left[\alpha \conv f_{1}\left(\frac{\cdot}{\alpha}\right)\right]\Box \left[
 (1-\alpha)\conv f_{2}\left(\frac{\cdot}{1-\alpha}\right)\right]+
\left[\alpha \frac{1}{\lambda}\phi\left(\frac{\cdot}{\alpha}\right)\right]\Box \left[
(1-\alpha)\frac{1}{\lambda}\phi\left(\frac{\cdot}{1-\alpha}\right)\right]
-\frac{1}{\lambda}\phi
\\
&=\left[\alpha \conv f_{1}\left(\frac{\cdot}{\alpha}\right)\right]\Box \left[
(1-\alpha)\conv f_{2}\left(\frac{\cdot}{1-\alpha}\right)\right].
\end{align*}

\ref{i:arithm}: For every $x\in\dom\phi$, we have $\averagef(x)$
\begin{align*}
 & \leq \alpha \conv \left(f_{1}+\frac{1}{\lambda}\phi\right)\left(\frac{\alpha x}{\alpha}\right)+
(1-\alpha) \conv \left(f_{2}+\frac{1}{\lambda}\phi\right)\left(\frac{(1-\alpha)x}{1-\alpha}\right)
-\frac{1}{\lambda}\phi
(x)\\
&\leq \alpha \left(f_{1}+\frac{1}{\lambda}\phi\right)\left(\frac{\alpha x}{\alpha}\right)+
(1-\alpha)\left(f_{2}+\frac{1}{\lambda}\phi\right)\left(\frac{(1-\alpha)x}{1-\alpha}\right)-\frac{1}{\lambda}\phi(x)\\
&=\alpha f_{1}(x)+\alpha\frac{1}{\lambda}\phi(x)+(1-\alpha)f_{2}(x)+(1-\alpha)\frac{1}{\lambda}\phi(x)-\frac{1}{\lambda}\phi(x)
=\alpha f_{1}(x)+(1-\alpha)f_{2}(x).
\end{align*}
\end{proof}

\begin{theorem}[epi-continuity II of Bregman proximal average]
Let {\bfseries A1--A5} hold. Define $\tilde{f}_{i}
:=f_i+\iota_{\dom\phi}$ for $i=1, 2$.
Then the following hold:
\begin{enumerate}
\item\label{i:m}
 For every $x\in\RR^n$, the function $\lambda\mapsto\averagef(x)$ is monotonically
decreasing on $]0,\bl [$.

\item\label{i:infinity}
$\lim_{\lambda\uparrow\bl}\averagef=
\left[\alpha\timess \conv \left(f_{1}+\frac{1}{\bl}\phi\right)\right]\Box \left[
(1-\alpha) \timess \conv \left(f_{2}+\frac{1}{\bl}\phi\right)\right]-\frac{1}{\bl}\phi$
pointwise.
In particular, for $\bl=+\infty$ one has
 $\lim_{\lambda\uparrow\infty}\averagef=\left[
\alpha\timess\conv\tilde{f}_{1}\right]\square\left[
(1-\alpha)\timess\conv\tilde{f}_{2}\right]$ pointwise;
consequently, $\averagef\epi\closu\left[
(\alpha\timess\conv\tilde{f}_{1})\square (
(1-\alpha)\timess\conv\tilde{f}_{2})
\right]$
as $\lambda\uparrow\infty$.

\item\label{i:zero:g}
$\lim_{\lambda\downarrow 0}\averagef=
\alpha{f}_{1}+
(1-\alpha){f}_{2}$ pointwise on $U$. Consequently, when $\dom f_{i}\subseteq \IDD$ for $i=1,2$,
$\averagef\epi
\alpha{f}_{1}+
(1-\alpha){f}_{2}$
as $\lambda\downarrow 0$.
\end{enumerate}
\end{theorem}

\begin{proof} We have $\averagef(x)=$
{\footnotesize
\begin{align}
&\inf\limits_{u+v=x}\bigg[\alpha\inf_{\sum_{i}\alpha_{i}x_{i}=\frac{u}{\alpha}
\atop{\sum_{i}\alpha_{i}=1,\alpha_i\geq0}}\sum_{i}\alpha_{i}\bigg(f_{1}(x_{i})+\frac{1}{\lambda}\phi(x_{i})\bigg)
+(1-\alpha)\inf_{\sum_{j}\beta_{j}y_{j}=\frac{v}{1-\alpha}\atop{\sum_{j}\beta_{j}=1,\beta_{j}\geq 0}}\sum_{j}\beta_{j}\bigg(f_{2}(y_{j})+\frac{1}{\lambda}\phi(y_{j})\bigg)\bigg]\nonumber\\
&\quad -\frac{1}{\lambda}\phi(x)\nonumber\\
&=\inf_{{\alpha\sum_i\alpha_{i}x_{i}+(1-\alpha)\sum_{j}\beta_{j}y_{j}=x}\atop{\sum_{i}\alpha_{i}=1,\sum_{j}\beta_{j}=1},
\alpha_{i}\geq 0, \beta_{j}\geq 0}\bigg[\alpha\sum_{i}\alpha_{i}f_{1}(x_{i})+(1-\alpha)\sum_{j}\beta_{j}f_{2}(y_{j})+
\nonumber\\
& \quad \frac{1}{\lambda}\underbrace{\bigg(\alpha\sum_{i}\alpha_{i}\phi(x_{i})+
(1-\alpha)\sum_{j}\beta_{j}\phi(y_{j})-
\phi(\alpha\sum_{i}\alpha_{i}x_{i}+(1-\alpha)\sum_{j}\beta_{j}y_{j})\bigg)}\bigg].\label{e:mono:l}
\end{align}}
\normalsize
The underbraced part is nonnegative because $\phi$ is convex,
$\sum_{i}\alpha_{i}=1,\sum_{j}\beta_{j}=1$, $\alpha_{i}, \beta_{j}\geq 0$.

\ref{i:m}: By \eqref{e:mono:l}, $\averagef$ is monotonically decreasing with respect to $\lambda$
 on $]0,\bl[$.

\ref{i:infinity}: From \ref{i:m} we obtain
$\lim_{\lambda\uparrow\bl}\averagef(x)=\inf_{\bl>\lambda>0}\averagef(x)=$
{\small
\begin{align}
&\inf_{\bl>\lambda>0}
\inf_{{\alpha\sum_i\alpha_{i}x_{i}+(1-\alpha)\sum_{j}\beta_{j}y_{j}=x}\atop{\sum_{i}\alpha_{i}=1,\sum_{j}
\beta_{j}=1},\alpha_{i}\geq 0, \beta_{j}\geq 0}\bigg[\alpha\sum_{i}\alpha_{i}f_{1}(x_{i})+(1-\alpha)\sum_{j}\beta_{j}f_{2}(y_{j})+\label{e:k1}\\
&\quad \frac{1}{\lambda}\bigg(\alpha\sum_{i}\alpha_{i}\phi(x_{i})+
(1-\alpha)\sum_{j}\beta_{j}\phi(y_{j})
-\phi(\alpha\sum_{i}\alpha_{i}x_{i}+(1-\alpha)\sum_{j}\beta_{j}y_{j})\bigg)\bigg]\nonumber\\
&=\inf_{{\alpha\sum_i\alpha_{i}x_{i}+(1-\alpha)\sum_{j}\beta_{j}y_{j}=x}\atop{\sum_{i}\alpha_{i}=1,\sum_{j}\beta_{j}=1},
\alpha_{i}\geq 0, \beta_{j}\geq 0}\inf_{\bl>\lambda>0}
\bigg[\alpha\sum_{i}\alpha_{i}f_{1}(x_{i})+(1-\alpha)\sum_{j}\beta_{j}f_{2}(y_{j})+\\
&\quad \frac{1}{\lambda}\bigg(\alpha\sum_{i}\alpha_{i}\phi(x_{i})+(1-\alpha)\sum_{j}\beta_{j}\phi(y_{j})
-
\phi(\alpha\sum_{i}\alpha_{i}x_{i}+(1-\alpha)\sum_{j}\beta_{j}y_{j})\bigg)\bigg]\nonumber\\
&=\inf_{{\alpha\sum_i\alpha_{i}x_{i}+(1-\alpha)\sum_{j}\beta_{j}y_{j}=x}\atop{\sum_{i}\alpha_{i}=1,\sum_{j}\beta_{j}=1},\alpha_{i}\geq 0, \beta_{j}\geq 0}\bigg[\alpha\sum_{i}\alpha_{i}f_{1}(x_{i})+(1-\alpha)\sum_{j}\beta_{j}f_{2}(y_{j})+\label{e:k2}\\
&\quad \frac{1}{\bl}\bigg(\alpha\sum_{i}\alpha_{i}\phi(x_{i})+(1-\alpha)\sum_{j}\beta_{j}\phi(y_{j})
-\phi(\alpha\sum_{i}\alpha_{i}x_{i}+(1-\alpha)\sum_{j}\beta_{j}y_{j})\bigg)\bigg]\nonumber\\
&=\left[\alpha\conv\bigg(f_{1}+\frac{1}{\bl}\phi\bigg)\bigg(\frac{\cdot}{\alpha}\bigg)\Box (1-\alpha)\conv\bigg(f_{2}+\frac{1}{\bl}\phi\bigg)\bigg(\frac{\cdot}{1-\alpha}\bigg)\right](x)
-\frac{1}{\bl}\phi(x)\nonumber.
\end{align}}
\normalsize
The above arguments also apply for $\bl=+\infty$.
The epi-convergence follows from \cite[Proposition 7.4(c)]{Rock98}.

\ref{i:zero:g}: By Theorem~\ref{complete}\ref{i:env:sum}, Proposition~\ref{p:phi:hull}\ref{i:h:e}
 and Theorem~\ref{t:squeeze},
on $U$ we have
$$\alpha f_{1}+(1-\alpha) f_{2}\geq \averagef\geq \alpha \lenv{\phi}{f_{1}}{\lambda}+(1-\alpha) \lenv{\phi}{f_{2}}{\lambda}.$$
The result follows by sending $\lambda$ to $0$ and applying Proposition~\ref{i:c}.

When $\dom f_{i}\subseteq\IDD$ for $i=1, 2$, we have
$\dom\averagef\subseteq \IDD$ by Theorem~\ref{complete}\ref{i:f:dom}.
Then
$\lim_{\lambda\downarrow 0}\averagef=
\alpha {f}_{1}+
(1-\alpha) {f}_{2}$ on $\RR^n$.
Because $\averagef$ is increasing as $\lambda\downarrow 0$, the $\epi$ follows from
\cite[Theorem 7.4(d)]{Rock98}.
\end{proof}

\section*{Acknowledgments}
The authors thank the editor M.\ Teboulle and two anonymous reviewers for helpful suggestions and
constructive feedback.
Xianfu Wang and Heinz Bauschke were partially supported by the Natural Sciences and
Engineering Research Council of Canada.

\end{document}